\documentclass[a4paper,11pt,twoside]{amsart}
\usepackage[english]{babel}
\usepackage[utf8]{inputenc}

\usepackage[a4paper,inner=3cm,outer=3cm,top=4cm,bottom=4cm,pdftex]{geometry}
\usepackage{fancyhdr}
\usepackage{color}
\usepackage{bold-extra}
\usepackage{ mathrsfs }
\usepackage{comment}
\usepackage{graphics}
\usepackage{aliascnt}
\usepackage[pdftex,citecolor=green,linkcolor=red]{hyperref}

\usepackage{amsmath}
\usepackage{amsfonts}
\usepackage{amssymb}
\usepackage{amsthm}
\usepackage{comment}
\usepackage{mathtools}
\usepackage{enumerate}
\usepackage{enumitem} 

\newtheorem{theorem}{Theorem}[section]
\newtheorem{corollary}[theorem]{Corollary}

\newtheorem{lemma}[theorem]{Lemma}
\newtheorem{proposition}[theorem]{Proposition}
\theoremstyle{definition}

\newtheorem{remark}[theorem]{Remark}
\newtheorem{conjecture}[theorem]{Conjecture}
\numberwithin{equation}{section}

\newcommand\eps{\varepsilon}

\renewcommand\P{\mathbf{P}}
\newcommand\p{\mathbf{p}}

\newcommand\E{\mathbf{E}}

\newcommand\R{\mathbb{R}}
\newcommand\Z{\mathbb{Z}}
\newcommand\N{\mathbb{N}}
\newcommand\C{\mathbb{C}}

\newcommand\supp{\mathrm{supp}}

\begin{document}

\title{Singmaster's conjecture in the interior of Pascal's triangle}

\author[Matom\"aki]{Kaisa Matom\"aki}
\address{Department of Mathematics and Statistics \\
University of Turku, 20014 Turku\\
Finland}
\email{ksmato@utu.fi}

\author[Radziwi{\l}{\l}]{Maksym Radziwi{\l}{\l}}
\address{ Department of Mathematics,
  Caltech, 
  1200 E California Blvd,
  Pasadena, CA, 91125 \\
	USA}
\email{maksym.radziwill@gmail.com}

\author[Shao]{Xuancheng Shao}
\address{Department of Mathematics, University of Kentucky\\
715 Patterson Office Tower\\
Lexington, KY 40506\\
USA}
\email{xuancheng.shao@uky.edu}

\author[Tao]{Terence Tao}
\address{Department of Mathematics, UCLA\\
405 Hilgard Ave\\
Los Angeles CA 90095\\
USA}
\email{tao@math.ucla.edu}

\author[Ter\"av\"ainen]{Joni Ter\"av\"ainen}
\address{Mathematical Institute, University of Oxford \\
Woodstock Road \\
Oxford OX2 6GG \\
United Kingdom}
\email{joni.teravainen@maths.ox.ac.uk}

\begin{abstract}  Singmaster's conjecture asserts that every natural number greater than one occurs at most a bounded number of times in Pascal's triangle; that is, for any natural number $t \geq 2$, the number of solutions to the equation $\binom{n}{m} = t$ for natural numbers $1 \leq m <n$ is bounded.  In this paper we establish this result in the interior region $\exp(\log^{2/3+\eps} n) \leq m \leq n - \exp(\log^{2/3+\eps} n)$ for any fixed $\eps>0$.  Indeed, when $t$ is sufficiently large depending on $\eps$, we show that there are at most four solutions (or at most two in either half of Pascal's triangle) in this region.  We also establish analogous results for the equation $(n)_m = t$, where $(n)_m \coloneqq n(n-1) \dots (n-m+1)$ denotes the falling factorial.
\end{abstract}

\maketitle

\section{Introduction}

In 1971, Singmaster \cite{singmaster} conjectured that any natural number greater than one only appeared in Pascal's triangle a bounded number of times.  In asymptotic notation\footnote{Our conventions for asymptotic notation are set out in Section \ref{notation-sec}.}, we can express this conjecture as

\begin{conjecture}[Singmaster's conjecture]\label{singmaster}  For any natural number $t \geq 2$, the number of integer solutions $1 \leq m < n$ to the equation 
\begin{equation}\label{bnt}
\binom{n}{m} = t
\end{equation}
 is $O(1)$.
\end{conjecture}

Note that we can exclude the edges $m=0,m=n$ of Pascal's triangle from consideration since $\binom{n}{m}=1$ in these cases.  Currently the largest known number of solutions to \eqref{bnt} for a given $t$ is eight, arising from $t=3003$ and 
\begin{equation}\label{nam}
(n,m) = (3003, 1), (78, 2), (15, 5), (14, 6), (14,8), (15, 10), (78, 76), (3003, 3002).
\end{equation}

For the purposes of attacking this conjecture, we may of course assume $t$ to be larger than any given absolute constant, which we shall implicitly do in the sequel.  In particular we can assume that the iterated logarithms
$$ \log_2 t \coloneqq \log\log t; \quad \log_3 t \coloneqq \log\log\log t$$
are well-defined and positive.

In view of the symmetry 
\begin{equation}\label{symmetry}
\binom{n}{m} = \binom{n}{n-m}
\end{equation}
we may restrict attention to the left half 
\begin{equation}\label{left-half}
\{ (m,n) \in \N \times \N: 1 \leq m \leq n/2 \}
\end{equation}
of Pascal's triangle.  For solutions to \eqref{bnt} in this half \eqref{left-half} of the triangle, we have
$$ t = \binom{n}{m} \geq \binom{2m}{m} \asymp 4^m / \sqrt{m}$$
by Stirling's approximation \eqref{stirling}, and thus we have the upper bound
\begin{equation}\label{m-bound}
 m \leq \frac{1}{\log 4} \log t + O( \log_2 t ).
\end{equation}
Since $n \mapsto \binom{n}{m}$ is an increasing function of $n$ for fixed $m \geq 1$, $n$ is uniquely determined by $m$ and $t$.  Thus by~\eqref{m-bound} we have at most $O(\log t)$ solutions to the equation $\binom{n}{m} = t$, a fact already observed in the original paper \cite{singmaster} of Singmaster.  This bound was improved to $O(\log t / \log_2 t)$ by Abbott, Erd\H{o}s, and Hansen \cite{aeh}, to $O( \log t \log_3 t / \log_2^2 t )$ by Kane \cite{kane-1}, and finally to $O( \log t \log_3 t / \log_2^3 t)$ in a followup work of Kane \cite{kane-2}. This remains the best known unconditional bound for the total number of solutions, although it was observed in \cite{aeh} that the improved bound $O_\eps(\log^{2/3+\eps} t)$ was available for any $\eps>0$ assuming the conjecture of Cram\'er \cite{cramer}.  

From the elementary inequalities
$$ \frac{(n-m)^m}{m!} < \binom{n}{m} \leq \frac{n^m}{m!}$$
and some rearranging we see that any solution to $\binom{n}{m} = t$ obeys the bounds
$$ (t m!)^{1/m} \leq n < (tm!)^{1/m} + m.$$
Applying Stirling's approximation \eqref{stirling} (and also $n \geq m$) we can thus obtain the order of magnitude of $n$ as a function of $m$ and $t$:
\begin{equation}\label{n-form}
n \asymp m t^{1/m}
\end{equation}
or equivalently
\begin{equation}\label{n-form-alt}
\frac{n}{m} \asymp \exp\left( \frac{\log t}{m} \right ).
\end{equation}
In particular we see that $n$ grows extremely rapidly when the ratio $m/\log t$ becomes small.  This makes the difficulty of the problem increase as $m / \log t$ approaches zero, and indeed treating the case of small values of $m/\log t$ is the main obstruction to making further progress on bounding the total number of solutions.

\begin{remark}  In the left half \eqref{left-half} of Pascal's triangle, a finer application of Stirling's approximation in \cite[(3.1)]{kane-1} gave the more precise estimate
$$ n = (tm!)^{1/m} + \frac{m-1}{2} + O( m t^{-1/m} ).$$
We will not explicitly use this estimate here.
\end{remark}

In this paper we study the opposite regime in which $m/\log t$ is relatively large, or equivalently (by \eqref{n-form-alt}) $n$ and $m$ are somewhat comparable (in the doubly logarithmic sense $\log_2 n \asymp \log_2 m$).  More precisely, we have the following result:

\begin{theorem}[Singmaster's conjecture in the interior of Pascal's triangle]\label{main}  Let $0 < \eps < 1$, and assume that $t$ is sufficiently large depending on $\eps$.  Then there are at most two solutions to \eqref{bnt} in the region $\exp( \log^{2/3 + \eps} n ) \leq m \leq n/2$.  By \eqref{symmetry}, we thus have at most four solutions to \eqref{bnt} in the region $\exp( \log^{2/3 + \eps} n ) \leq m \leq n - \exp( \log^{2/3 + \eps} n )$.   Furthermore, in the smaller region $\exp( \log^{2/3 + \eps} n ) \leq m \leq n/\exp(\log^{1-\eps'} n)$ there is at most one solution, whenever $0 < \eps' < \frac{\eps}{2/3+\eps}$ and $t$ is sufficiently large depending on both $\eps$ and $\eps'$.
\end{theorem}

\begin{remark}\label{fib} The bound of two (or four) solutions is absolutely sharp, in view of the infinite family of solutions observed in \cite{lind}, \cite{singmaster-2}, \cite{tovey} to the equation
$$ \binom{n+1}{m+1} = \binom{n}{m+2}$$
given by $n = F_{2j+2} F_{2j+3} - 1$, $m = F_{2j} F_{2j+3}-1$, where $F_j$ denotes the $j^{th}$ Fibonacci number.  See also \cite{jenkins} for further analysis of equations of this type.  Besides this infinite family of collisions, and the ``trivial'' ones generated by \eqref{symmetry}, $\binom{n}{0} = 1$, and $\binom{n}{m} = \binom{\binom{n}{m}}{1}$, the only further known collisions between binomial coefficients arise from the identities $\binom{n}{2} = \binom{n'}{m'}$ for
$$ (n,n',m') = (16,10,3), (21, 2, 4), (52, 22, 3), (120, 36, 3), (153, 19, 5), (221, 17, 8)$$
as well as the example in \eqref{nam}.
It was conjectured by de Weger \cite{deweger} that these above examples generate all the non-trivial collisions $\binom{n}{m} = \binom{n'}{m'} = t$; this would of course imply Singmaster's conjecture.  This conjecture has been verified for $(m,m') = (2,3)$ \cite{avanesov}, for $(m,m') = (2,4)$ \cite{pinter}, \cite{deweger-1}, for $(m,m') = (2,5)$ \cite{bugeaud}, for $(m,m') = (3,4)$ \cite{mordell}, \cite{deweger}, and $(m,m') = (2, 6), (2, 8), (3, 6), (4, 6), (4, 8)$ \cite{stroeker}, and for $n \leq 10^6$ or $t \leq 10^{60}$ in \cite{bbw}.
\end{remark}

\begin{remark}
In view of Theorem \ref{main}, we now see that to prove Conjecture \ref{singmaster}, we may restrict attention without loss of generality to the region $2 \leq m \leq \exp(\log^{2/3+\eps} n)$  for any fixed $\eps>0$, or equivalently (by \eqref{n-form-alt}) to $2 \leq m \leq \frac{\log t}{\log^{3/2-\eps}_2 t}$ for any fixed $\eps>0$.  It follows from the conjecture of de Weger \cite{deweger} mentioned in Remark \ref{fib} that for $t$ sufficiently large there is only at most one solution in this region, that is to say all but a finite number of binomial coefficients $\binom{n}{m}$ for $2 \leq m \leq \exp(\log^{2/3+\eps} n)$ are distinct.  In this direction, the number of solutions to the equation $\binom{n}{m} = \binom{n'}{m'}$ for fixed $2 \leq m < m'$ has been shown (via Siegel's theorem on integral points) to be finite in \cite{bst} (see also the earlier result \cite{kiss} treating the case $(m,m')=(2,p)$ for an odd prime $p$).    This implies that there are no collisions in the regime $2 \leq m \leq w(n)$ if $w$ is a function of $n$ that goes to infinity sufficiently slowly as $n \to \infty$.  Unfortunately, due to the reliance on Siegel's theorem, the function $w$ given by these arguments is completely ineffective.
\end{remark}

\begin{remark}  For some previous bounds of this type, in \cite{aeh} it was shown that the number of solutions to \eqref{bnt} in the range $n^{5/6} \leq m \leq n/2$ was $O(\log^{3/4} t)$, while the arguments in \cite[\S 7]{kane-1}, after some manipulation, show that the number of solutions to \eqref{bnt} in the range $\exp(\log^{1/2+\eps} n) \leq m \leq n^{5/6}$ is $O_\eps( \log t / \log_2^3 t )$.
\end{remark}

\begin{remark} The implied quantitative bounds in the hypothesis ``$t$ is sufficiently large depending on $\eps$'' are effective; however, we have made no attempt whatsoever to optimize them in this paper, and will likely be too large to be of use in numerical verification of Singmaster's conjecture in their current form.
\end{remark}

\subsection{An analog for falling factorials} The methods used to handle the equation \eqref{bnt} can be modified to treat the variant equation
\begin{equation}\label{falling}
 (n)_m = t
\end{equation}
for integers $1 \leq m < n$ and $t \geq 2$, where $(n)_m$ denotes the falling factorial
$$ (n)_m \coloneqq n (n-1) \dots (n-m+1) = m! \binom{n}{m}.$$
We exclude the cases $m=0,m=n$ since $(n)_0 = 1$ and $(n)_n = (n)_{n-1} = n!$.
In \cite[Theorem 4]{aeh} it was shown that for any $t \geq 2$ the number of integer solutions $(m,n)$ to \eqref{falling} with $1 \leq m \leq n-1$ is $O( \sqrt{\log t})$.  We do not directly improve upon this bound here, but can obtain an analogue of Theorem \ref{main}:

\begin{theorem}[Falling factorial multiplicity in the interior]\label{main-falling}  Let $0 < \eps < 1$, and assume that $t$ is sufficiently large depending on $\eps$.  Then there are at most two integer solutions to \eqref{falling} in the region $\exp( \log^{2/3 + \eps} n ) \leq m < n$.
\end{theorem}

We establish this result in Section \ref{falling-sec}.  Note that the bound of two is best possible, as can be seen from the infinite family of solutions
$$ (a^2 - a)_{a^2 - 2a} = (a^2 - a - 1)_{a^2 - 2a + 1}$$
for any integer $a>2$, and more generally
$$ ( (a)_b)_{(a)_b - a} = ( (a)_b - 1 )_{(a)_b - a + b - 1}$$
whenever $2 \leq b < a$ are integers.

\subsection{Strategy of proof}

Theorem \ref{main} is a consequence of two Propositions that we now describe. The proof of Theorem \ref{main-falling} will follow a similar pattern as described here and we refer the reader to Section \ref{falling-sec} for details. 

\begin{proposition}[Distance estimate]\label{distance} Let $\varepsilon > 0$. Suppose we have two solutions $(n,m), (n',m')$ to \eqref{bnt} in the left half \eqref{left-half} of Pascal's triangle.  Then one has
$$ m' - m \ll_\eps \exp( \log^{2/3+\eps}(n+n') )$$
for any $\eps>0$.  Furthermore, if
$$ m, m' \geq \exp( \log^{2/3+\eps}(n+n') )$$
then we additionally have
$$ n' - n \ll_\eps \exp( \log^{2/3+\eps}(n+n') ).$$
\end{proposition}

Note how this proposition is consistent with the example in Remark \ref{fib}. We shall discuss the proof of Proposition~\ref{distance} in Section~\ref{ss:proofmethods}. For the application to Theorem \ref{main}, Proposition \ref{distance} localizes all solutions to \eqref{bnt} to a region of small diameter.  To conclude Theorem \ref{main}, we can now proceed by adapting the Taylor expansion arguments of Kane \cite{kane-1}, \cite{kane-2}, in which one views $n$ as an analytic function of $m$ (keeping $t$ fixed) and exploits the non-vanishing of certain derivatives of this function; see Section \ref{analytic-sec}. This is what the proposition below accomplishes. In fact in our analysis only two derivatives of this function are needed (i.e., we only need to exploit the convexity properties of $n$ as a function of $m$). 

\begin{proposition}[Kane-type estimate] \label{prop:kane}
  Let $\varepsilon > 0$. 
  Suppose that $(n,m)$ is a solution to \eqref{bnt} in the left-half \eqref{left-half} of Pascal's triangle. There there exists at most one other solution $(n',m') \neq (n,m)$ to \eqref{bnt} with $m' < m$, $n' > n$ and
  $$
  |m - m'| + |n - n'| \ll \exp((\log_2 t)^{1 - \varepsilon}).
  $$
\end{proposition}

With these two Propositions at hand it is easy to deduce Theorem \ref{main}. 

\begin{proof}[Deduction of Theorem \ref{main}]
Let $\eps>0$, let $t$ be sufficiently large depending on $\eps$, and let $(n,m)$ be the solution to \eqref{bnt} in the region
\begin{equation}\label{region}
\{ (n,m): \exp(\log^{2/3+\eps} n) \leq m \leq n/2\}
\end{equation}
with the maximal value of $m$ (if there are no such solutions then of course Theorem \ref{main} is trivial).  For brevity we allow all implied constants in the following arguments to depend on $\eps$.  If $(n',m')$ is any other solution in this region, then $m' < m$ and $n' > n$.  From \eqref{n-form-alt} we have
$$
 m \gg \frac{\log t}{\log n} \geq \frac{\log t}{\log^{\frac{1}{2/3+\eps}} m} \gg \frac{\log t}{\log_2^{\frac{1}{2/3+\eps}} t}
$$
thanks to \eqref{m-bound}.  From further application of \eqref{n-form-alt} we then have
$$ n \ll \exp( O( \log_2^{\frac{1}{2/3+\eps}} t ) ).$$
Similarly for $n'$.  Applying Proposition \ref{distance} (with $\eps$ replaced by a sufficiently small quantity), we conclude that
\begin{equation}\label{mim}
 m-m', n'-n \ll_{\eps'} \exp( O( \log^{1-\eps'}_2 t ) )
\end{equation}
whenever $1-\eps' > \frac{2/3}{2/3+\eps}$, or equivalently $\eps' < \frac{\eps}{2/3 + \eps}$.
The result now follows from Proposition \ref{prop:kane}.
\end{proof}

\begin{remark} The above arguments showed that for $t$ sufficiently large depending on $\eps$, there were at most four solutions to \eqref{bnt} in the region $\exp(\log^{2/3+\eps} n) \leq m \leq n - \exp(\log^{2/3+\eps} n)$.  A modification of the argument also shows that there cannot be exactly \emph{three} such solutions.  For if this were the case, we see from \eqref{symmetry} that there must be a solution $(n,m)$ with $n=2m$, so that $m \asymp \log t$ by Stirling's approximation.  For all other solutions $(n',m')$ to \eqref{bnt} we have $n' \geq n+1$, hence
$$ \binom{n}{n/2} = t = \binom{n'}{m'} \geq \binom{n+1}{m'}$$
and hence (by Stirling's approximation)
$$ \binom{n+1}{m'} \leq \left( \frac{1}{2} + O\left(\frac{1}{n}\right)\right) \binom{n+1}{(n+1)/2}.$$
By Stirling's approximation (or the central limit theorem of de Moivre and Laplace) this forces $|m' - \frac{n+1}{2}| \gg \sqrt{n}$, thus $|m'-m| \gg m^{1/2}$.  But this contradicts \eqref{mim}.
\end{remark}

\subsection{Proof methods}
\label{ss:proofmethods}

We now discuss the method of proof of Proposition \ref{distance}, which is our main new contribution.
In contrast to the ``Archimedean'' arguments of Kane (such as Proposition \ref{prop:kane}) that use real and complex analysis of the binomial coefficients $\binom{n}{m}$, the proof of Proposition \ref{distance} relies more on ``non-Archimedean'' arguments, based on evaluating the $p$-adic valuations $v_p\left( \binom{n}{m} \right)$ for various primes $p$, defined as the number of times $p$ divides $\binom{n}{m}$.  From the classical Legendre formula
\begin{equation}\label{legendre} 
v_p( n! ) = \sum_{j=1}^\infty \left\lfloor \frac{n}{p^j} \right\rfloor,
\end{equation}
where $\lfloor x \rfloor$ is the integer part of $x$, we see that
\begin{equation}\label{np}
\begin{split}
v_p\left( \binom{n}{m} \right) &= \sum_{j=1}^\infty \left(\left\lfloor \frac{n}{p^j} \right\rfloor - \left\lfloor \frac{m}{p^j} \right\rfloor - \left\lfloor \frac{n-m}{p^j} \right\rfloor \right)\\
&= \sum_{j=1}^\infty \left(\left\{ \frac{m}{p^j} \right\} + \left\{ \frac{n-m}{p^j} \right\} - \left\{ \frac{n}{p^j} \right\} \right)
\end{split}
\end{equation}
where $\{x\} \coloneqq x - \lfloor x \rfloor$ denotes the fractional part of $x$.  Note that the summands here vanish whenever $p^j > n$.  From this identity we see that if $(n,m), (n',m')$ are two solutions to \eqref{bnt} then we must have
\begin{equation}\label{p-eq}
\sum_{j=1}^\infty \left(\left\{ \frac{m}{p^j} \right\} + \left\{ \frac{n-m}{p^j} \right\} - \left\{ \frac{n}{p^j} \right\}\right)  = \sum_{j=1}^\infty \left(\left\{ \frac{m'}{p^j} \right\} + \left\{ \frac{n'-m'}{p^j} \right\} - \left\{ \frac{n'}{p^j} \right\} \right)
\end{equation}
for all primes $p$.  Our strategy will be to apply this equation with $p$ set equal to a \emph{random} prime $\p$ drawn uniformly amongst all primes in the interval $[P, P+P\log^{-100} P]$ where the scale $P$ is something like $\exp( \log^{2/3+\eps/2}(n+n'))$, and inspect the distribution of the resulting random variables on the left and right-hand sides of \eqref{p-eq} in order to obtain a contradiction when $m,m'$ or $n,n'$ are sufficiently well separated.  In order to do this we need some information concerning the equidistribution of fractional parts such as $\{ \frac{n}{\p^j}\}$.  This will be provided by the following estimate, proven in Section \ref{equid-sec}. There and later the letter $p$ always denotes a prime.

\begin{proposition}[Equidistribution estimate]\label{equid}  Let $\eps>0$ and $P \geq 2$ and let $I$ be an interval contained in $[P,2P]$.  Let $M,N$ be real numbers with $M,N = O( \exp(\log^{3/2-\eps} P) )$, and let $j$ be a natural number. 
\begin{itemize}
\item[(i)]
For all $A > 0$,
\[
 \sum_{p \in I} e\left( \frac{N}{p} + \frac{M}{p^j} \right) = \int_I e\left( \frac{N}{t} + \frac{M}{t^{j}}\right)\ \frac{dt}{\log t} + O_{\eps,A}( P\log^{-A} P ).
\]
\item[(ii)] 
Let $W \colon \R^2 \to \C$ be a smooth $\Z^2$-periodic function.  Then, for all $A > 0$,
$$ \sum_{p \in I} W\left( \frac{N}{p}, \frac{M}{p^j} \right) = \int_I W\left( \frac{N}{t}, \frac{M}{t^{j}} \right)\ \frac{dt}{\log t} + O_{\eps,A}( \|W\|_{C^{3}} P \log^{-A} P ),$$
where
$$ \|W\|_{C^{3}} \coloneqq \sum_{j=0}^{3} \sup_{x \in \R^2} |\nabla^j W(x)|.$$
\end{itemize}
\end{proposition}

One can generalize this proposition to control the joint equidistribution of any bounded number of expressions of the form $\{ \frac{n}{p^j}\}$, but for our applications it will suffice to understand the equidistribution of pairs $\{\frac{N}{p}\}$, $\{\frac{M}{p^j}\}$. 

When it comes to the proof of Proposition \ref{equid}, the first step is to use Fourier expansion to reduce part (ii) of the proposition to part (i).  For part (i), the case where $\frac{|N|}{P}+\frac{|M|}{P^j}$ is small (say $\leq \log^{O(A)} P$) is easily handled using the prime number theorem with classical error term. In the regime where $\frac{|N|}{P}+\frac{|M|}{P^j}$ is large, we use Vaughan's identity to decompose the sum in (i) into type I and II sums, and assert that these exhibit cancellation; the type I and II bounds are given in \eqref{type-i} and \eqref{type-ii}.

Both type I and type II sums can be handled using Vinogradov's bound for sums of the form $\sum_{n\in I}e(f(n))$ with $f$ smooth, although we need to first cut from $I$ small intervals around zeros of the first $\log P$ derivatives of $N/t+M/t^j$. This way we obtain that the sum in (i) exhibits cancellation. It is here that the restriction $N,M=O(\exp(\log^{3/2-\varepsilon} P))$ arises; even under the Riemann hypothesis we do not know how to relax this requirement\footnote{Using standard randomness heuristics one could tentatively conjecture that this restriction $N,M=O(\exp(\log^{3/2-\varepsilon} P))$ could be relaxed to $N,M = O( \exp(P^c) )$ for some constant $c>0$; this would improve the range $\exp(\log^{2/3+\eps} n) \leq m \leq n/2$ in Theorem \ref{main} to $\log^C n \leq m \leq n/2$ for some constant $C>0$.}.

Once the equidistribution estimate, Proposition \ref{equid}, is established, the analysis of the distribution of both sides of \eqref{p-eq} is relatively straightforward, as long as the scale $P$ is chosen so that the powers $P^j$ do not lie close to various integer combinations of $m,n,m',n'$.  However, there are some delicate cases when two of the numbers $n, m, n-m, n', m', n'-m'$ are ``commensurable'' in the sense that one of them is close to a rational multiple of the other, where the rational multiplier has small height.  Commensurable integers are also known to generate some exceptional examples of integer factorial ratios \cite{bober}, \cite{bober-2}, \cite{sound-ratio}.  Fortunately, we can handle these cases in our context by an analysis of covariances between various fractional parts $\{ \frac{n_1}{\p} \}, \{ \frac{n_2}{\p} \}$, in particular taking advantage of the fact that these covariances are non-negative up to small errors, and small unless $n_1,n_2$ are very highly commensurable.

\subsection{Acknowledgments}

KM was supported by Academy of Finland grant no. 285894. MR acknowledges the support of NSF grant DMS-1902063 and a Sloan Fellowship. XS was supported by NSF grant DMS-1802224. TT was supported by a Simons Investigator grant, the James and Carol Collins Chair, the Mathematical Analysis \& Application Research Fund Endowment, and by NSF grant DMS-1764034. JT was supported by a Titchmarsh Fellowship. 

\subsection{Notation}\label{notation-sec}

We use $X \ll Y$, $X = O(Y)$, or $Y \gg X$ to denote the estimate $|X| \leq CY$ for some constant $C$.  If we wish to permit this constant to depend on one or more parameters we shall indicate this by appropriate subscripts, thus for instance $O_{\eps,A}(Y)$ denotes a quantity bounded in magnitude by $C_{\eps,A} Y$ for some quantity $C_{\eps,A}$ depending only on $\eps,A$.  We write $X \asymp Y$ for $X \ll Y \ll X$.

We use $1_E$ to denote the indicator of an event $E$, thus $1_E$ equals $1$ when $E$ is true and $0$ otherwise.  

We let $e$ denote the standard real character $e(x) \coloneqq e^{2\pi ix}$.
\section{Derivative estimates}\label{analytic-sec}

We generalize the binomial coefficient $\binom{n}{m}$ to real $0 \leq m \leq n$ by the formula
$$\binom{n}{m} \coloneqq \frac{\Gamma(n+1)}{\Gamma(m+1)\Gamma(n-m+1)}$$
where 
$$\Gamma(x) \coloneqq \frac{e^{-\gamma x}}{x} \prod_{n=1}^\infty 	\left(1+\frac{x}{n}\right)^{-1} e^{x/n}$$
is the Gamma function (with $\gamma$ the Euler--Mascheroni constant).  This is of course consistent with the usual definition of the binomial coefficient. Observe that the digamma function 
$$\psi(x) \coloneqq \frac{\Gamma'}{\Gamma}(x) = -\gamma + \sum_{n=0}^\infty \frac{1}{n+1} - \frac{1}{n+x}$$
is a smooth increasing concave function on $(0,+\infty)$, with 
$$ \psi'(x) = \sum_{n=0}^\infty \frac{1}{(n+x)^2}$$
positive and decreasing, and
$$ \psi''(x) = -\sum_{n=0}^\infty \frac{2}{(n+x)^3}$$
negative.  For future reference we also observe the standard asymptotics
\begin{align}
 \psi(x) &= \log x + O\left(\frac{1}{x}\right) \label{psi-1}\\
 \psi'(x) &= \frac{1}{x} + O\left(\frac{1}{x^2}\right) \label{psi-2}\\
 \psi''(x) &= -\frac{1}{x^2} + O\left(\frac{1}{x^3}\right) \label{psi-3}
\end{align}
and the Stirling approximation
\begin{equation}\label{stirling}
\log \Gamma(x) = x \log x - x - \frac{1}{2} \log x + \log \sqrt{2\pi} + O\left(\frac{1}{x}\right)
\end{equation}  
for any $x \geq 1$; see e.g., \cite[\S 6.1, 6.3, 6.4]{abramowitz+stegun}.  One could also extend these functions meromorphically to the entire complex plane, but we will not need to do so here.

From the increasing nature of $\psi$ we see that $n \mapsto \binom{n}{m}$ is strictly increasing on $[m,+\infty)$ for fixed real $m > 0$, and from Stirling's approximation \eqref{stirling} we see that it goes to infinity as $n \to \infty$.  Thus for given $t > 1$, we see from the inverse function theorem that there exists a unique smooth function $f_t \colon [0,+\infty) \to [0,+\infty)$ with $f_t(m) > m$ for all $m$, such that
\begin{equation}\label{bnm-2}
 \binom{f_t(m)}{m} = t.
\end{equation}
In particular, the equation \eqref{bnt} holds for given integers $1 \leq m \leq n$ and $t \geq 2$ if and only if $n = f_t(m)$.
This function $f_t$ was analyzed by Kane \cite{kane-1}, who among other things was able to extend $f_t$ holomorphically to a certain sector, which then allowed him to estimate high derivatives of this function.  However, for our analysis we will only need to control the first few derivatives of $f_t$, which can be estimated by hand:

\begin{proposition}[Estimates on the first few derivatives]\label{derivi} Let $t,m$ be sufficiently large with $m \leq f_t(m)/2$.  Then
\begin{equation}\label{fatm}
 f_t(m) \asymp m t^{1/m}
\end{equation}
and
\begin{equation}\label{one}
-f'_t(m) \asymp (f_t(m) - 2m) \frac{\log t}{m^2}
\end{equation}
and
\begin{equation}\label{ten}
 f''_t(m) \asymp f_t(m) \left(\frac{\log t}{m^2}\right)^2.
\end{equation}
In particular, $f_t$ is convex and decreasing in this regime.
\end{proposition}

The bound \eqref{fatm} can be viewed as a generalization of \eqref{n-form} to non-integer values of $n,m,t$.

\begin{proof}  Taking logarithms in \eqref{bnm-2} we have
\begin{equation}\label{logf}
 \log \Gamma(f_t(m)+1) - \log \Gamma(f_t(m)-m+1) - \log \Gamma(m+1) = \log t.
\end{equation}
Writing $n=f_t(m) \geq 2m$, we thus see from the mean value theorem that
$$ m \psi(n - \theta m + 1) - \log \Gamma(m+1) = \log t$$
for some $0 \leq \theta \leq 1$ depending on $t,m$.  Applying \eqref{psi-1}, we conclude that
$$ \log(n - \theta m) = \frac{1}{m} ( \log t + \log \Gamma(m+1) ) + O( \frac{1}{n} )$$
which implies that
$$ n \asymp n - \theta m \asymp \exp( \frac{1}{m} ( \log t + \log \Gamma(m+1) ) )$$
and the claim \eqref{fatm} then follows from Stirling's approximation \eqref{stirling}.

If we differentiate \eqref{logf} we obtain
\begin{equation}\label{deriv}
 f'_t(m) \psi( f_t(m)+1) - (f'_t(m)-1) \psi(f_t(m)-m+1) - \psi(m+1) = 0.
\end{equation}
In particular we obtain the first derivative formula
\begin{equation}\label{form}
f'_t(m) = \frac{\psi(m+1) - \psi(n-m+1)}{\psi(n+1) - \psi(n - m + 1)}.
\end{equation}
From \eqref{psi-2} and the mean value theorem we have
\begin{equation}\label{psini}
 \psi(n+1) - \psi(n - m + 1) \asymp \frac{m}{n}
\end{equation}
while from either the mean-value theorem and \eqref{psi-2} (if $m \asymp n$) or from \eqref{psi-1} (if say $m \leq n/4$) we see that
$$ \psi(n-m+1) - \psi(m+1) \asymp \frac{n-2m}{n} \log \frac{n}{m}.$$
We conclude that
$$ -f'_t(m) \asymp \frac{n-2m}{m} \log \frac{n}{m}$$
and the claim \eqref{one} follows from~\eqref{fatm}.

Differentiating \eqref{deriv} again, we conclude
$$
 f''_t(m) \psi( n+1) + (f'_t(m))^2 \psi'( n+1) - f''_t(m) \psi(n-m+1) 
-(f'_t(m)-1)^2 \psi'(n-m+1)  - \psi'(m+1) = 0.
$$
which we can rearrange using \eqref{form} as
\begin{align*}
f''_t(m) (\psi(n+1) - \psi(n-m+1))^3 &= (\psi(n+1) - \psi(n-m+1))^2 \psi'(m+1) \\
&\quad + (\psi(n+1)-\psi(m+1))^2 \psi'(n-m+1)\\
&\quad - (\psi(n-m+1)-\psi(m+1))^2 \psi'(n+1).
\end{align*}
From \eqref{psini}, \eqref{fatm} it thus suffices to show that
\begin{align*}
 (\psi(n+1) - \psi(n-m+1))^2 \psi'(m+1) \quad &  \\
+ (\psi(n+1)-\psi(m+1))^2 \psi'(n-m+1) \quad & \\
- (\psi(n-m+1)-\psi(m+1))^2 \psi'(n+1) &\asymp
\frac{m}{n^2} \log^2(n/m).
\end{align*}
The quantity $(\psi(n+1) - \psi(n-m+1))^2 \psi'(m+1)$ is non-negative and is of size $O( m/n^2 )$ by \eqref{psini}, \eqref{psi-2}.  Thus it will suffice to show that
$$ (\psi(n+1)-\psi(m+1))^2 \psi'(n-m+1) - (\psi(n-m+1)-\psi(m+1))^2 \psi'(n+1) \asymp
\frac{m}{n^2} \log^2(n/m).$$
We split the left-hand side as the sum of
$$ (\psi(m+1)-\psi(n+1))^2 (\psi'(n-m+1) - \psi'(n+1)) $$
and
\begin{align*}
&\psi'(n+1) [(\psi(n+1)-\psi(m+1))^2 - (\psi(n-m+1)-\psi(m+1))^2] \\
&=(\psi(n+1)-\psi(m+1) + \psi(n-m+1)-\psi(m+1)) (\psi(n+1)-\psi(n-m+1)) \psi'(n+1).
\end{align*}
From \eqref{psi-1}, \eqref{psi-3}, and the mean value theorem the first term is positive and comparable to $\frac{m}{n^2} \log^2 \frac{n}{m}$; similarly, from \eqref{psi-1}, \eqref{psi-2}, and~\eqref{psini} the second term is positive and bounded above by $O( \frac{m}{n^2} \log \frac{n}{m} )$.  The claim follows.
\end{proof}

To apply these derivative bounds, we use the following lemma that implicitly appears in \cite{kane-1}, \cite{kane-2}:

\begin{lemma}[Small non-zero derivative implies few integer values]\label{integ}  Let $k \geq 1$ be a natural number, and suppose that $f: I \to \R$ is a smooth function on an interval $I$ of some length $|I|$ such that one has the derivative bound
\begin{equation}\label{deriv-f}
 0 < \left| \frac{1}{k!} f^{(k)}(x) \right| < |I|^{-k(k+1)/2}
\end{equation}
for all $x \in I$.  Then there are at most $k$ integers $m \in I$ for which $f(m)$ is also an integer.
\end{lemma}

\begin{proof}  Suppose for contradiction that there are $k+1$ distinct integers $m_1,\dots,m_{k+1} \in I$ with $f(m_1),\dots,f(m_{k+1})$ an integer.  By Lagrange interpolation, the function
\begin{equation}\label{p-def}
 P(x) \coloneqq \sum_{i=1}^{k+1} \prod_{1 \leq j \leq k+1: j \neq i} \frac{x-m_j}{m_i - m_j} f(m_i)
\end{equation}
is a polynomial of degree at most $k$ such that $f(x)-P(x)$ vanishes at $m_1,\dots,m_{k+1}$.  By many applications of Rolle's theorem (see \cite[Corollary 2.1]{kane-1}), there must then exist $x_* \in I$ such that $f^{(k)}(x_*) - P^{(k)}(x_*)$ vanishes.  From \eqref{p-def}, $\frac{1}{k!} P^{(k)}(x)$ (which is the degree $k$ coefficient of $P(x)$) is an integer multiple of $\frac{1}{\prod_{1 \leq i < j \leq k+1} |m_i-m_j|} \geq |I|^{-k(k+1)/2}$, and thus either vanishes or has magnitude at least $|I|^{-k(k+1)/2}$.  But this contradicts \eqref{deriv-f}.
\end{proof}

As an application of these bounds, we can locally control the number of solutions \eqref{bnt} in the region $n^{1/2+\eps} \leq m \leq n/2$, thus giving a version of Theorem \ref{main} in a small interval:

\begin{corollary}\label{core}  Let $0 < \eps < 1$, let $t$ be sufficiently large depending on $\eps$, and suppose that $(n,m)$ is a solution to \eqref{bnt} in the left half \eqref{left-half} of Pascal's triangle with $m \geq n^{1/2+\eps}$.  Then there is at most one other solution $(n',m')$ to \eqref{bnt} in the interval $m' \in [m - m^{\eps/10}, m]$.
\end{corollary}

\begin{proof}  From \eqref{n-form-alt} and the hypothesis $n^{1/2+\eps} \leq m \leq n/2$ we have
\begin{equation}\label{man}
 \frac{\log t}{\log_2 t} \ll m \ll \log t.
\end{equation}
For $x$ in the interval $I \coloneqq [m - m^{\eps/10},m]$, we then have $\frac{\log t}{x} = \frac{\log t}{m} + O( m^{-2+\eps/10} \log t) = \frac{\log t}{m}  + O(1)$, and so we see from Proposition \ref{derivi} and \eqref{man} that $f_t(x) \asymp n$ and
$$
0 < |f''_t(x)| \ll n \left( \frac{\log t}{m^2}\right)^2 \ll \frac{n}{m^2} \log_2^2 t \ll \frac{n}{m^2} \log^2 m$$
for all $x \in I$.  Since $m \geq n^{1/2+\eps}$ and $t$ is sufficiently large depending on $\eps$, $m$ is also sufficiently large depending on $\eps$, and we have
$$
0 < |f''_t(x)| < |I|^{-3}$$
for all $x \in I$.  Applying Lemma \ref{integ}, there are at most two integers $m' \in I$ with $f_t(m')$ an integer.  Since $m$ is already one of these integers, the claim follows.
\end{proof}

The same method, using higher derivative estimates on $f_t$, also gives similar results (with weaker bounds on the number of solutions) for $m < n^{1/2+\eps}$; see \cite{kane-1}, \cite{kane-2}.  However, we will only need to apply this method in the $m \geq n^{1/2+\eps}$ regime here.

We are now ready to prove Proposition \ref{prop:kane}.
\begin{proof}[Proof of Proposition \ref{prop:kane}]
Let $\eps>0$, let $t$ be sufficiently large depending on $\eps$, and let $(n,m)$ be a solution to \eqref{bnt} in the region
\begin{equation}\label{region1}
\{ (n,m): \exp(\log^{2/3+\eps} n) \leq m \leq n/2\}
\end{equation}
For brevity we allow all implied constants in the following arguments to depend on $\eps$.  Suppose $(n',m')$ is another solution in this region with $m' < m$, $n' > n$ and$$
 m-m', n'-n \ll_{\eps'} \exp( O( \log^{1-\eps'}_2 t ) ).
$$
From \eqref{one} and convexity (and the bounds $m \ll \log t$ and $m-m' \geq 1$) we have
\begin{align*}
n'-n &= f_t(m')-f_t(m) \\
&\geq f'_t(m) (m'-m) \\
&\gg (n-2m) \frac{\log t}{m^2} (m-m') \\
&\gg \frac{n-2m}{m} \\
&= \frac{n}{m} - 2
\end{align*}
and thus
$$ n/m \ll_{\eps'} \exp( O( \log^{1-\eps'}_2 t ) )$$
From \eqref{n-form-alt} we have $n \gg \log t$, hence $\log^{1-\eps'}_2 t \ll \log^{1-\eps'} n$, and so for some constant $C > 0$, 
$m \geq n / \exp(C \log^{1-\eps'} n) \geq n^{9/10}$ (shrinking $\eps'$ slightly if necessary) if $t$ is sufficiently large depending on $\eps'$. The result now follows from Corollary \ref{core}.

%giving the second claim of Theorem \ref{main}. From \eqref{mim} we now conclude that all solutions $(n',m')$ to \eqref{bnt} in the region \eqref{region1} lie in (say) $[m - m^{1/100}, m]$.
\end{proof}

It remains to establish Proposition \ref{distance}.  This will be the objective of the next two sections of the paper.

\section{The distance bound}\label{distance-sec}

In this section we assume Proposition \ref{equid} and use it to establish Proposition \ref{distance}.

Throughout this section $0 < \eps < 1$ will be fixed; we can assume it to be small.  We may assume that $t$ is sufficiently large depending on $\eps$, as the claim is trivial otherwise.  We may assume that $m' < m$, hence also $n' > n$.  We assume for sake of contradiction that
at least one of the claims
\begin{equation}\label{meri}
 m - m' \geq \exp( \log^{2/3 + \eps} n' )
\end{equation}
and
\begin{equation}\label{meri-2}
 m, m', n'-n \geq \exp( \log^{2/3 + \eps} n' )
\end{equation}
is true, as the claim is trivial otherwise.  This allows us to select a ``good'' scale:

\begin{lemma}[Selection of scale]\label{scale}  With the above assumptions, there exists $P > 1$ obeying the following axioms:
\begin{itemize}
\item[(i)] ($m,m',n,n'$ not too large) We have $m,m',n,n' \leq \exp( \log^{\frac{3}{2} - \frac{\eps}{10}} P )$.  (In particular, $P$ will be sufficiently large depending on $\eps$, since otherwise $t = O_\eps(1)$.)
\item[(ii)]  (Dichotomy)  If $a,a',b,b'$ are integers with $|a|, |a'|, |b|, |b'| \leq \log^{1/100} P$, and $j$ is a natural number, then either
\begin{equation}\label{latch}
 |a m + a'm' + bn + b'n'| \leq P^j / \log^{1000} P
\end{equation}
or
$$ |a m + a'm' + bn + b'n'| \geq P^j \log^{1000} P.$$
\item[(iii)]  (Separation)  At least one of the statements
$$ m-m' \geq P \log^{100} P$$
and
$$ m, m', n'-n \geq P \log^{100} P$$
is true.
\end{itemize}
\end{lemma}

\begin{proof}  We restrict $P$ to be a power of two in the range
$$\exp( \log^{2/3+\eps/2} n' ) \leq P \leq \exp(2 \log^{2/3+\eps/2} n' );$$
such a choice will automatically obey (i) since $n' > n > m > m'$ and (iii) since we assumed that either~\eqref{meri} or~\eqref{meri-2} holds. There are $\gg \log^{2/3+\eps/2} n'$ choices for $P$.  Some of these will not obey (ii), but we can control the number of exceptions as follows.  Firstly, observe that the conclusion \eqref{latch} will hold unless $j = O( \log^{1/3} n' )$, so we may restrict attention to this range of $j$.  The number of possible tuples $(a,a',b,b',j)$ is then 
$O( \log^{4/100} P \log^{1/3} n' )$.  For each such tuple, we see from the restriction on $P$ that the number of $P$ with
$$ P^j / \log^{1000} P < |a m + a'm' + bn + b'n'| < P^j \log^{1000} P$$
is at most $O(\log_2 n')$ (since $a m + a'm' + bn + b'n'$ is of size $O( (n')^2 )$, say).  Thus we see that the total number of $P$ which fail to obey (ii) is at most
$$ O( \log^{4/100} P \log^{1/3} n' \log_2 n' ) $$
which is negligible compared to the total number of choices, which is $\gg \log^{2/3+\eps/2} n'$.  Thus we can find a choice of $P$ which obeys all of (i), (ii), and (iii), giving the claim.
\end{proof}

Henceforth we fix a scale $P$ obeying the properties in Lemma \ref{scale}.
We now introduce a relation $\approx$ on the reals by declaring $x \approx y$ if $|x-y| \leq P / \log^{1000} P$.  Thus, by Lemma \ref{scale}(ii), if $am + a'm' + bn + b'n' \not \approx 0$ for $a,a',b,b'$ as in Lemma \ref{scale}(ii) then $|am + a'm' + bn + b'n'| \geq P \log^{1000} P$.  Also, from Lemma \ref{scale}(iii), at least one of the statements
$$ m \not \approx m'$$
and
$$ m, m', n'-n \not \approx 0$$
is true.

We introduce a random variable $\p$, which is drawn uniformly from the primes in the interval $I \coloneqq [P, P + P \log^{-100} P]$ (note that there is at least one such prime thanks to the prime number theorem).  From \eqref{p-eq} we surely have
$$ \sum_{j=1}^\infty \left(\left\{ \frac{m}{\p^j} \right\} + \left\{ \frac{n-m}{\p^j} \right\} - \left\{ \frac{n}{\p^j} \right\}\right)  = \sum_{j=1}^\infty \left(\left\{ \frac{m'}{\p^j} \right\} + \left\{ \frac{n'-m'}{\p^j} \right\} - \left\{ \frac{n'}{\p^j} \right\}\right).$$
We can restrict attention to those $j$ with $j \leq \log^{1/2} P$, since the summands vanish otherwise.
For any real number $N$, we may take covariances of both sides of this identity with the random variable $\{ \frac{N}{\p}\}$ to conclude that
\begin{equation}\label{jp}
 \sum_{j \leq \log^{1/2} P} \left(c_j( N, m) + c_j(N, n-m) - c_j(N, n)\right) = \sum_{j \leq \log^{1/2} P} \left(c_j( N, m') + c_j(N, n'-m') - c_j(N, n')\right)
\end{equation}
for any real number $N$, where the covariances $c_j(N,M)$ are defined as
\begin{align*}
c_j(N,M) &\coloneqq \E \left\{ \frac{N}{\p}\right\} \left\{ \frac{M}{\p^j}\right\} - \E \left\{ \frac{N}{\p}\right\} \E \left\{ \frac{M}{\p^j}\right\} \\
&\coloneqq \E \left(\frac{1}{2} - \left\{\frac{N}{\p}\right\}\right) \left( \frac{1}{2} - \left\{ \frac{M}{\p^j}\right\} \right) - \E \left(\frac{1}{2} - \left\{ \frac{N}{\p}\right\} \right) \E \left( \frac{1}{2} - \left\{ \frac{M}{\p^j}\right\}  \right).
\end{align*}

We now compute these covariances:

\begin{proposition}[Covariance estimates]\label{covar}  Let $N, M \in \{ m, n, m-n, m', n', n'-m'\}$, and $j$ be a natural number with $1 \leq j \leq \log^{1/2} P$.
\begin{itemize}
\item[(i)]  If $j \geq 2$, then $c_j(N,M) \ll \log^{-10} P$.
\item[(ii)]  If $j=1$ and $N \approx 0$ or $M \approx 0$, then $c_j(N,M) \ll \log^{-1000} P$.
\item[(iii)]  If $j=1$, $N, M \not \approx 0$ and there exist coprime natural numbers $1 \leq a,b \leq \log^{1/100} P$ such that $aN \approx bM$, then
$c_j(N,M) = \frac{1}{12 ab} + O( \log^{-1/1000} P)$.
\item[(iv)]  If $j=1$ and $N,M$ are not of the form in (ii) or (iii), then $c_j(N,M) \ll \log^{-1/1000} P$.
\end{itemize}
\end{proposition}

\begin{remark} The term $\frac{1}{12ab}$ appearing in Proposition \ref{covar}(iii) is also the covariance between $\{n{\bf x}\}$ and $\{m{\bf x}\}$ for ${\bf x}$ drawn randomly from the unit interval whenever $n,m$ are natural numbers with $an=bm$ for some coprime $a,b$; see \cite[Section 2]{sound-preprint}.  Indeed, both assertions are proven by the same Fourier-analytic argument, and Proposition \ref{covar} endows the linear span of the six functions $\{ \frac{N}{\p} \}$ for $N \in \{m,n,m-n,m',n',n'-m'\}$ with an inner product closely related to the norm $N()$ studied in \cite{sound-preprint}, the structure of which is the key to obtaining a contradiction from our separation hypotheses on $n-n', m-m'$.
\end{remark}

\begin{proof}[Proof of Proposition~\ref{covar} assuming Proposition~\ref{equid}]
We first dispose of the easy case (ii).  If $N \approx 0$, then $\{ \frac{N}{\p} \} \leq \log^{-1000} P$, and the claim follows from the triangle inequality; similarly if $M \approx 0$ or actually if $M \leq P^j/\log^{1000} P$. Hence by Lemma \ref{scale}(ii), we may from now on assume that 
\[
N \geq P \log^{1000} P \quad \text{and} \quad M \geq P^j \log^{1000} P.
\]  

To handle the remaining cases we use the truncated Fourier expansion
\begin{equation}
\label{eq:1/2-xFourier}
\begin{split}
\frac{1}{2} - \{ x \} &= \sum_{0 < |n| \leq N_0} \frac{e(n x)}{2\pi i n} + O\left(\frac{1}{1+N_0 \mathrm{dist}(x, \Z)}\right) \\
&= \sum_{0 < |n| \leq N_0} \frac{e(n x)}{2\pi i n} + O\left(1_{\mathrm{dist}(x, \Z) \leq N_0^{-1/2}} + \frac{1}{N_0^{1/2}}\right)
\end{split}
\end{equation}
that holds for any $N_0 \geq 1$ (see e.g.~\cite[Formula (4.18)]{ik}).

Our primary tool is Proposition \ref{equid}. Note that, for $t \in I$, $\log t = \log P + O(\log^{-99} P)$, so that together with the prime number theorem Proposition~\ref{equid} implies that
\begin{equation}\label{eq}
 \E W\left( \frac{N}{\p}, \frac{M}{\p^{j}} \right) = \frac{1}{|I|} \int_I W\left( \frac{N}{t}, \frac{M}{t^j} \right)\ dt + O_{\eps}( \|W\|_{C^3} \log^{-99} P)
\end{equation}
for any smooth $\Z^2$-periodic $W \colon \R^2 \to \C$ and that, for any $M', N' = O(\exp(\log^{3/2-\varepsilon/2} P))$,
\begin{equation}\label{eq:EevsII}
\E e\left( \frac{N'}{\p} + \frac{M'}{\p^{j}} \right) = \frac{1}{|I|} \int_I e\left( \frac{N'}{t}+ \frac{M'}{t^j} \right)\ dt + O_{\eps}(\log^{-99} P)
\end{equation}

Applying \eqref{eq} with $W$ a suitable cutoff localized to the region $\{ (x,y): \mathrm{dist}(x, \Z) \leq 2 N_0^{-1/2} \}$ that equals one on $\{ (x,y): \mathrm{dist}(x, \Z) \leq N_0^{-1/2} \}$ chosen so that $\|W\|_{C^3} \ll N_0^{3/2}$, we see that, for any $N_0 \in [1, \log^{20} P]$ we have
$$
\P\left(\mathrm{dist}\left( \frac{N}{\p}, \Z\right) \leq N_0^{-1/2}\right)
\ll \frac{1}{|I|} \int_I 1_{\mathrm{dist}( \frac{N}{t}, \Z) \leq 2N_0^{-1/2}}\ dt + N_0^{-1/2}.$$
Since $N \geq P \log^{1000} P$, the first term on the right-hand side can be computed to be $O( N_0^{-1/2})$.  Thus
\begin{equation}
\label{eq:PdistNsmall}
\P\left(\mathrm{dist}\left( \frac{N}{\p}, \Z\right) \leq N_0^{-1/2}\right) \ll N_0^{-1/2}
\end{equation}
and a similar argument gives
\begin{equation}
\label{eq:PdistMsmall}
\P\left(\mathrm{dist}\left( \frac{M}{\p^j}, \Z\right) \leq N_0^{-1/2}\right) \ll N_0^{-1/2}.
\end{equation}

To prepare for the proofs of parts (i), (iii) and (iv), let us first show that, for $1 \leq j \leq \log^{1/2} P$, we have
\begin{equation}
\label{m-ast}
\E \left( \frac{1}{2} - \left\{\frac{M}{\p^j} \right\} \right) \ll \log^{-10} P.
\end{equation}
We use the Fourier expansion~\eqref{eq:1/2-xFourier} with $N_0 = \log^{20} P$. Averaging over $p \in I$ and applying~\eqref{eq:PdistMsmall} to handle the first error term, we see that
$$ \E \left( \frac{1}{2} - \left\{\frac{M}{\p^j} \right\} \right) = \sum_{0 < |m| \leq \log^{20} P} \frac{1}{2 \pi i m} \E e\left( m \frac{M}{\p^j} \right) + O(\log^{-10} P).$$
By the triangle inequality and \eqref{eq:EevsII}, it suffices to show that, for every non-zero integer $m = O(\log^{20} P)$,
$$ \frac{1}{|I|} \int_I e\left( m \frac{M}{t^j} \right)\ dt \ll \log^{-11} P.$$
Recalling that $M \geq P^j \log^{1000} P$, this estimate follows from a standard integration by parts (see e.g.~\cite[Lemma 8.9]{ik}).
Similarly
\begin{equation}
\label{n-ast}
\E \left(\frac{1}{2} - \left\{\frac{N}{\p} \right\}\right) \ll \log^{-10} P .
\end{equation}

Furthermore, using similarly~\eqref{eq:1/2-xFourier},~\eqref{eq:PdistNsmall},~\eqref{eq:PdistMsmall} and~\eqref{eq:EevsII}, we see that, whenever $1 \leq N_0 \leq \log^{20} P$,
\begin{equation}
\label{eq:NpMpFour}
\E \left(\frac{1}{2} - \left\{\frac{N}{\p} \right\}\right) \left(\frac{1}{2} - \left\{\frac{M}{\p^j} \right\} \right) = -\sum_{0 < |m|, |n| < N_0} \frac{1}{4\pi^2 mn} \frac{1}{|I|} \int_I e\left( n \frac{N}{t} + m \frac{M}{t^j} \right)\ dt + O\left(\frac{1}{N_0^{1/2}}\right).
\end{equation}

Now we are ready to prove (i), (iii), and (iv). Let us start with (i). In light of~\eqref{m-ast},~\eqref{n-ast} and~\eqref{eq:NpMpFour} with $N_0 = \log^{20} P$, it suffices to show that
$$ \frac{1}{|I|} \int_I e\left( n \frac{N}{t} + m \frac{M}{t^j} \right)\ dt \ll \log^{-11} P.$$
whenever $n,m = O(\log^{20} P)$ are non-zero integers. Applying a change of variables $t = P / s$, we reduce to showing that
\begin{equation}
\label{eq:sintclaim}
\int_{1/(1 + \log^{-100} P)}^1 e( as + bs^j )\ ds \ll \log^{-200} P
\end{equation}
(say), where $a \coloneqq nN/P$ and $b \coloneqq mM/P^j$.  By hypothesis, we have $|a|, |b| \geq \log^{1000} P$.  Since $2 \leq j \leq \log^{1/2} P$, the derivative $a+jbs^{j-1}$ of the phase $as+bs^j$ is at least $\log^{200} P$ outside of an interval of length at most $O(\log^{-200} P)$, and~\eqref{eq:sintclaim} now follows from a standard integration by parts (see e.g.~\cite[Lemma 8.9]{ik}).  This concludes the proof of (i).

Let us now turn to (iv). In light of~\eqref{m-ast},~\eqref{n-ast} and~\eqref{eq:NpMpFour} with $N_0 = \log^{1/500} P$, it suffices to show that 
$$ \frac{1}{|I|} \int_I e\left( \frac{nN+mM}{t} \right)\ dt \ll \log^{-1/500} P$$
whenever $n,m = O( \log^{1/500} P)$ are non-zero integers.
From the hypothesis (iv) and Lemma \ref{scale}(ii) (after factoring out any common multiple of $n$ and $m$), we have $|nN+mM| \geq P \log^{1000} P$.  The claim (iv) now follows from integration by parts.

Finally we show (iii).   In light of~\eqref{m-ast},~\eqref{n-ast} and~\eqref{eq:NpMpFour} with $N_0 = \log^{1/500} P$, it suffices to show that 
$$ -\sum_{0 < |n|,|m| \leq \log^{1/500} P} \frac{1}{4\pi^2 mn} \frac{1}{|I|} \int_I e\left( \frac{nN+mM}{t} \right)\ dt = \frac{1}{12ab} + O( \log^{-1/1000} P ).$$
Let us first consider those $n,m = O(\log^{1/500} P)$ for which $nN+mM \not \approx 0$. By Lemma \ref{scale}(ii) $|nN+mM| \geq P \log^{1000} P$ and similarly to case (iv), the contribution of such pairs $(n,m)$ is acceptable. 

Consider now the case $nN \approx -mM$ for some non-zero integers $n,m = O(\log^{1/500} P)$. By assumption also $aN \approx bM$ for some co-prime positive integers $a, b \leq \log^{1/100} P$. and hence by Lemma~\ref{scale}(ii) $-amM \approx bnM$ which contradicts the assumption $M \not \approx 0$ unless $(n,m)$ is a multiple of $(a,-b)$. On the other hand if $(n, m)$ is a multiple of $(a, -b)$, then $n N \approx -mM$ by Lemma~\ref{scale}(ii).

Thus it remains to show that
$$ \sum_{0 <|k| \leq \frac{\log^{1/500} P}{\max\{a,b\}}} \frac{1}{4\pi^2 k^2 ab} 
\frac{1}{|I|} \int_I e\left(\frac{kaN-kbM}{t} \right)\ dt = \frac{1}{12ab} + O( \log^{-1/1000} P ).$$
Since $aN \approx bM$ we have, for every $k \leq \log^{1/500} P$,
$$ \frac{1}{|I|} \int_I e\left(\frac{kaN -kbM}{t} \right)\ dt = 1 - O( \log^{-100} P )$$
and so it suffices to show that
$$ \sum_{0 <|k| \leq \frac{\log^{1/500} P}{\max\{a, b\}}} \frac{1}{4\pi^2 k^2 ab} = \frac{1}{12ab} + O( \log^{-1/1000} P )$$
This is trivial for $ab \geq \log^{1/1000} P$.  For $ab \leq \log^{1/1000} P$ the claim follows from the Basel identity
$$ \sum_{k=1}^\infty \frac{1}{k^2} = \frac{\pi^2}{6}$$
and the tail bound
$$ \sum_{k \geq \log^{1/1000} P} \frac{1}{k^2} \ll \log^{-1/1000} P.$$
\end{proof}

Now we can get back to proving Proposition~\ref{distance} assuming Proposition~\ref{equid}. From Proposition \ref{covar}(i) and \eqref{jp} we see that
\begin{equation}\label{c1}
 c_1( N, m) + c_1(N, n-m) - c_1(N, n) = c_1( N, m') + c_1(N, n'-m') - c_1(N, n') + O( \delta)
\end{equation}
for $N \in \{m,n,n-m,m',n',m'-n'\}$, where for brevity we introduce the error tolerance
$$ \delta \coloneqq \log^{-1/1000} P.$$
We can now arrive at the desired contradiction by some case analysis (reminiscent of that in \cite{sound-preprint, sound-ratio}) using the remaining portions of Proposition \ref{covar}, as follows.

\subsection*{Case $m' \approx 0$}  Applying \eqref{c1} with $N = m$, we conclude from Proposition \ref{covar}(ii) that
\begin{equation}
\label{eq:covN=m} 
c_1(m,m) + c_1(m,n-m) - c_1(m,n) = c_1(m,n'-m') - c_1(m,n') + O(\delta).
\end{equation}
From Lemma \ref{scale}(iii) we have $m \not \approx 0$ (and hence also $n-m, n'-m',n' \not \approx 0$, since these quantities are greater than or equal to $m$), hence by Proposition \ref{covar}(iii) we have $c_1(m,m) = \frac{1}{12} + O(\delta)$. Furthermore, since $m' \approx 0$, we see from Lemma~\ref{scale}(ii) that, for $1 \leq a, b \leq \log^{1/100} P$, $am \approx b(n'-m')$ if and only if $am \approx bn'$. Hence Proposition~\ref{covar}(iii) (iv) implies that 
\[
c_1(m, n'-m') = c_1(m, n') + O(\delta).
\] 
Plugging these facts into~\eqref{eq:covN=m} and rearranging, we obtain
\[
\frac{1}{12} + c_1(m,n-m) = c_1(m,n) + O(\delta).
\]
But by Proposition~\ref{covar}(iii), (iv) we know that $c_1(m, n-m) \geq -O(\delta)$, so that 
\[
c_1(m, n) \geq \frac{1}{12} + O(\delta)
\]
But since $m \not \approx n$ (because $m \leq n/2$ and $m \not \approx 0$), another application of Proposition~\ref{covar}(iii), (iv) 
gives
\[
c_1(m, n) \leq \frac{1}{2}\frac{1}{12} + O(\delta),
\]
which is a contradiction.

Since $m'$ was the smallest element of $\{m,n,n-m,m',n',m'-n'\}$, we now thus have $N \not \approx 0$ for all $N \in \{m,n,n-m,m',n',m'-n'\}$, and case (ii) of Proposition \ref{covar} no longer applies.

\subsection*{Case $m \not \approx m'$ and $m' \not \approx 0$}  We apply \eqref{c1} with $N = m'$ to conclude that
\begin{equation}
\label{eq:mnotapm'1st}
c_1(m',m) + c_1(m',n-m) - c_1(m',n) = c_1(m', m') + c_1(m',n'-m') - c_1(m',n') + O( \delta ).
\end{equation}
Now if there are no co-prime positive integers $a, b \leq \log^{1/100} P$ such that $am' \approx bn'$ or $am' \approx b(n'-m')$, then by Proposition~\ref{covar}(iv) we have 
\[
c_1(m',n'-m') - c_1(m',n') = O(\delta)
\]
On the other hand, if such co-prime integers exist, then $am' \approx bn'$ if and only if $(a-b)m' \approx b(n'-m')$ and necessarily $a > b$, so that by Proposition \ref{covar}(iii) we have in this case
\begin{equation}
 c_1(m',n'-m') - c_1(m',n') = \frac{1}{12 (a-b) b} - \frac{1}{12 ab} - O(\delta) \geq -O(\delta).
\end{equation}

Since Proposition~\ref{covar}(iii) also gives $c_1(m', m') \geq 1/12 + O(\delta)$, combining with~\eqref{eq:mnotapm'1st} we obtain that
\begin{equation}\label{c11}
c_1(m',m) + c_1(m',n-m) - c_1(m',n) \geq \frac{1}{12} - O( \delta ).
\end{equation}
On the other hand, since $m' \not \approx m$, we also have $m' \not \approx n-m$ since $n-m \geq m > m'$.  By Proposition \ref{covar}(iii), (iv), we have
$$ c_1(m',m) + c_1(m',n-m) \leq \frac{1}{12} \cdot \frac{1}{2} + \frac{1}{12} \cdot \frac{1}{2} + O(\delta),$$
which can be improved to
\begin{equation}
\label{eq:mnappm'imp}
c_1(m',m) + c_1(m',n-m) \leq \frac{1}{12} \cdot \frac{1}{3} + \frac{1}{12} \cdot \frac{1}{2} + O(\delta),
\end{equation}
unless both $m \approx 2m'$ and $n-m \approx 2m'$. Since by Proposition \ref{covar}(iii), (iv) we have
$c_1(m',n) \geq - O( \delta )$, the estimate~\eqref{eq:mnappm'imp} contradicts \eqref{c11}. 

Hence we must have both $m \approx 2m'$ and $n-m \approx 2m'$.  But then Lemma \ref{scale}(ii) forces $n \approx 4m'$, hence by Proposition \ref{covar}(iii)
$$ c_1(m',m) + c_1(m',n-m) - c_1(m',n) = \frac{1}{12}\cdot \frac{1}{2}+ \frac{1}{12} \cdot\frac{1}{2} - \frac{1}{12} \cdot\frac{1}{4} + O(\delta),$$
and we again contradict \eqref{c11}. 

\subsection*{Case $m \approx m'$ and $m' \not \approx 0$} By Lemma \ref{scale}(iii), we must have $n \not \approx n'$.  We apply \eqref{c1} for $N=n$ to obtain
\begin{equation}\label{can}
 c_1( n, m) + c_1(n, n-m) - c_1(n, n) = c_1( n, m') + c_1(n, n'-m') - c_1(n, n') + O( \delta).
\end{equation}
Since $m \approx m'$, we have by Proposition~\ref{covar}(iii), (iv) (using also Lemma~\ref{scale}(ii)) that $c_1(n, m) = c_1(n, m') + O(\delta)$. Proposition~\ref{covar}(iii) also gives $c_1(n, n) = 1/12 + O(\delta)$. Plugging these into~\eqref{can} and rearranging, we obtain
\begin{equation}
\label{simplmappm'}
c_1(n, n-m) + c_1(n, n') = \frac{1}{12} + c_1(n, n'-m') + O( \delta).
\end{equation}
Since $n \not \approx n'$ and $m \not \approx 0$, we see from Proposition \ref{covar}(iii), (iv) that 
\begin{equation}
\label{simplmappm'-2}
c_1(n, n-m) + c_1(n, n') \leq \frac{1}{12}\cdot \frac{1}{2}+\frac{1}{12}\cdot\frac{1}{2} + O(\delta)
\end{equation}
which can be improved to 
\begin{equation}
\label{simplmappm'-3}
c_1(n, n-m) + c_1(n, n') \leq \frac{1}{12}\cdot\frac{1}{3}+\frac{1}{12}\cdot\frac{1}{2} + O(\delta)
\end{equation}
unless $2(n-m) \approx n$ and $n' \approx 2n$. Now~\eqref{simplmappm'-3} contradicts~\eqref{simplmappm'} since by Proposition \ref{covar}(iii), (iv) $c_1(n, n'-m') \geq - O(\delta)$. 

Hence we can assume that $2(n-m) \approx n$ and $n' \approx 2n$. But using $m \approx m'$ and Lemma~\ref{scale}(ii) this implies that $2(n'-m') \approx 3n$, so that by~\eqref{simplmappm'} and Proposition~\ref{covar}(iii) we obtain
\[
c_1(n, n-m) + c_1(n, n') = \frac{1}{12} + c_1(n, n'-m') + O( \delta) = \frac{1}{12} + \frac{1}{12} \cdot \frac{1}{2\cdot 3} + O( \delta).
\]
contradicting~\eqref{simplmappm'-2}.

\begin{remark} Morally speaking, the ability to obtain a contradiction here reflects the fact that one cannot have an identity of the form
\begin{equation}\label{mx}
 \{ mx \} + \{ (n-m) x \} - \{ nx\} = \{ m'x\} + \{ (n'-m')x\} - \{n' x\}
\end{equation}
for almost all real numbers $x$ and some integers $1 \leq m \leq n/2$, $1 \leq m' \leq n'/2$ unless one has both $m=m'$ and $n=n'$ (this type of connection goes back to Landau \cite[p. 116]{landau-1}).  This latter fact is easily established by inspecting the jump discontinuities of both sides of \eqref{mx}, but it is also possible to establish it by computing the covariances of both sides of \eqref{mx} with $\{ Nx\}$ for various choices of $N$, and the arguments above can be viewed as an adaptation of this latter method.
\end{remark}

It remains to establish Proposition \ref{equid}.  This will be established in the next section.
\section{Equidistribution}\label{equid-sec}

In this section we prove Proposition \ref{equid}.  Fix $\eps,A$.  We may assume that $P$ is sufficiently large depending on $\eps,A$, as the claim is trivial otherwise.  If we have $P^j \geq M \log^A P$ then we can replace in both parts of the proposition $\frac{M}{P^j}$ by $0$ with negligible error, so we may assume that either $M=0$ or $P^j < M \log^A P$.  In either event we may thus assume that $j \leq \log^{1/2} P$.  Next, by partitioning $I$ into at most $\log^{100} P$ intervals of length at most $P \log^{-100} P$ and using the triangle inequality, it suffices (after suitable adjustment of $P$, $A$) to assume that $I \subset [P, P + P \log^{-100} P]$.  In particular we have
\begin{equation}\label{pj}
 P^{j-1} \leq t^{j-1} \leq 2 P^{j-1}
\end{equation}
for all $t \in I$.

Let us first reduce Proposition~\ref{equid}(ii) to Proposition \ref{equid}(i). We perform a Fourier expansion
$$ W(x,y) = \sum_{n,m \in \Z} c_{n,m} e(nx + my)$$
where by integration by parts the Fourier coefficients
$$ c_{n,m} = \int_{\R^2/\Z^2} W(x,y) e(-nx-my)\ dx dy$$
obey the bounds
$$ |c_{n,m}| \ll \|W\|_{C^3} (1 + |n| + |m|)^{-3}.$$
By the triangle inequality, the contributions of those frequencies $n,m$ with $|n| + |m| \geq \log^{2A} P$ is then acceptable.  By a further application of the triangle inequality, Proposition~\ref{equid}(ii) follows from showing that
$$ 
\sum_{p \in I} e\left( n \frac{N}{p} + m \frac{M}{p^j} \right) = \int_I e\left( n \frac{N}{t} + m \frac{M}{t^{j}}\right)\ \frac{dt}{\log t} + O_{\eps,A}( P\log^{-10A} P )$$
whenever $n,m$ are integers with $|n| + |m| \leq \log^{2A} P$. But this follows from Proposition \ref{equid}(i) by adjusting the values of $\varepsilon, A, M, N$ suitably.

The proof of part (i) will use the standard tools of Vaughan's identity and Vinogradov's exponential sum estimates.  We state a suitable form of the latter tool here:

\begin{lemma}[Vinogradov's exponential sum estimate]\label{le_vin_expsum} Let $X\geq 2$, $F \geq X^4$, and $\alpha \geq 1$. Let $I\subset [X,2X]$ be an interval. Let $f(x)$ be a smooth function on $I$ satisfying for all $t\in I$ 
\begin{align}\label{vin24}
 \alpha^{-r^3}F\leq \frac{t^r}{r!}|f^{(r)}(t)|\leq\alpha^{r^3}F   
\end{align}
for all integers $1\leq r\leq 10\lceil \log F/(\log X)\rceil+1.$
Assume further that
\begin{align}\label{vin23}
(\log \alpha)\frac{(\log F)^2}{(\log X)^3}<10^{-3}.    
\end{align}
Then we have
\begin{align}\label{vin22}
\sum_{n\in I}e(f(n))\ll \alpha X\exp(-2^{-18}(\log X)^3/(\log F)^2),
\end{align}
where the implied constant is absolute.
\end{lemma}

\begin{proof} This is essentially \cite[Theorem 8.25]{ik} with minor modifications (the modification needed is that we only assume \eqref{vin24} for $r$ in a certain range, not all integers $r\geq 1$.).

Let $R:=10\lceil \log F/(\log X)\rceil$, and as in \cite[p. 217]{ik}, let
\begin{align*}
F_n(q):=\sum_{0\leq r\leq R}\alpha_r(n)q^r,\quad \alpha_r(n) \coloneqq \frac{f^{(r)}(n)}{r!}.    
\end{align*}
 Let $S_f(I)$ denote the sum in \eqref{vin22}. By Taylor's formula, for any $q\geq 1$ we have
\begin{align*}
S_f(I)=\sum_{n\in I}e(F_n(q))+O\left(q+Xq^{R+1}\frac{\max_{t\in I}|f^{(R+1)}(t)|}{(R+1)!}\right).    
\end{align*}
Let $\mathcal{Q}:=\{xy:1\leq x\leq V,1\leq y\leq V\}\cap \mathbb{N}$, where $\mathcal{Q}$ is interpreted as a multiset. Also let $Q=|\mathcal{Q}|=V^2$. Then
\begin{align*}
 S_f(I)=\sum_{n\in I}|\mathcal{Q}|^{-1}\sum_{q\in \mathcal{Q}}e(F_n(q))+O\left(Q+XQ^{R+1}\frac{\max_{t\in I}|f^{(R+1)}(t)|}{(R+1)!}\right).   
\end{align*}
We take $V=X^{1/4}$ in which case by \eqref{vin24} the error term is 
\begin{equation}
\label{eq:TaylorError}
\begin{split}
&\ll X^{1/2}+X \cdot X^{(R+1)/2}\alpha^{(R+1)^3}F/X^{R+1}\\
&\ll X^{1/2}+X^{-(R+1)/4}\alpha^{(R+1)^3} \cdot (F X^{1-(R+1)/4}).
\end{split}
\end{equation}
The term in the parenthesis is $\leq FX^{3/4} F^{-10/4} \leq 1$. Using also~\eqref{vin23} we see that~\eqref{eq:TaylorError} is $\ll X^{1/2}$ which is in particular smaller than the right-hand side of \eqref{vin22}. The sum $\sum_{q\in \mathcal{Q}}e(F_n(q))$ is precisely the one estimated in \cite[pp. 217--225]{ik}. The only assumption needed of $f$ in that argument is \eqref{vin24}, and the only restriction on $F$ and $X$ there is $F\geq X^4$. Hence, we conclude that the lemma holds by following the analysis there verbatim.
\end{proof}

We now apply this estimate to obtain an estimate for an exponential sum over integers.

\begin{proposition}[Exponential sums over integers]\label{expi}  Let $\eps > 0$, $A \geq 1$, $X \geq 2$, $2 \leq j \ll \log^{1/2} X$, and let $N,M$ be real numbers with $N,M \ll \exp( O( \log^{3/2-\eps} X))$.  Let $I$ be an interval in $[X,X+X\log^{-100} X]$.  Then
\begin{equation}\label{integer-sum}
\sum_{n \in I} e\left( \frac{N}{n} + \frac{M}{n^j} \right) \ll_{\eps,A} X (1+F)^{-c} \log^{O(A)} X + X \log^{-A} X 
\end{equation}
for some absolute constant $c>0$, where
$$ F \coloneqq \frac{|N|}{X} + \frac{|M|}{X^j}.$$

\end{proposition}

\begin{proof} We may assume without loss of generality that $A$ is sufficiently large, and $X$ is sufficiently large depending on $\eps,A$.  By hypothesis we have $F \ll \exp( O( \log^{3/2-\eps} X))$.
We may assume that $F \geq \log^{CA} X$ for a large absolute constant $C$, since the claim is trivial otherwise.  

Let $f \colon I \to \R$ denote the phase function
$$ f(t) \coloneqq \frac{N}{t} + \frac{M}{t^j}.$$
Then for any $r \geq 1$ and $t \in I$ we have
\begin{equation}
\label{eq:expint1} 
\frac{t^r}{r!}|f^{(r)}(t)| = \left| \frac{N}{t} + \frac{M_r}{t^j} \right| \asymp X^{-1} |N + M_r / t^{j-1}|
\end{equation}
where
$$ M_r \coloneqq \binom{r+j-1}{j-1} M.$$
Since
$$ 1 \leq \binom{r+j-1}{j-1} \leq (r+j)^r = \exp(r\log(r+j)) = \exp(O( r^2 \log_2 X ) ) $$
we conclude that
$$ M_r = \exp( O( r^2 \log_2 X ) ) M$$
and
$$ \frac{|N|}{X} + \frac{|M_r|}{X^j} = \exp( O( r^2 \log_2 X ) ) F.$$

If $|M_r| \leq |N| X^{j-1}/4$ then from the triangle inequality and \eqref{pj} we have 
$$
X^{-1} |N + M_r / t^{j-1}| \asymp X^{-1} |N| = \exp( O( r^2 \log_2 X ) ) F.
$$
Consider then the case $|M_r| > |N| X^{j-1}/4$. We have the upper bound
$$ X^{-1} |N + M_r / t^{j-1}| \ll \frac{|M_r|}{X^j} \ll \exp( O( r^2 \log_2 X ) ) F $$
for all $t \in I$ from the triangle inequality. Furthermore, since the function $t \mapsto -1/t^{j-1}$ has derivative $\asymp j/X^j$ on $I$, we also have, for all $t$ outside of an interval of length $O( X \log^{-2A} X)$, the lower bound
$$ X^{-1} |N + M_r / t^{j-1}| \gg \frac{|M_r|}{X^j} \log^{-3A} X \gg \exp( O( r^2 \log_2 X ) ) F \log^{-3A} X.$$
 If we set $\alpha \coloneqq \log^{4A} X$ and $A$ is sufficiently large, then we conclude from~\eqref{eq:expint1} and the bounds above that the estimate \eqref{vin24} holds for all $1 \leq r \leq \log X$ and all $t \in I$ outside the union of $O(\log X)$ intervals of length $O( X \log^{-2A} X)$.  The contribution of these exceptional intervals to \eqref{integer-sum} is negligible, and removing them splits $I$ up into at most $O(\log X)$ subintervals, so by the triangle inequality it suffices to show that
$$
\sum_{n \in I'} e\left( \frac{N}{n} + \frac{M}{n^j} \right) \ll_{\eps,A} X \log^{-2A} X $$
for any subinterval $I'$ with the property that \eqref{vin24} holds for all $t \in I'$ and $1 \leq r \leq \log X$.  If $F \geq X^4$, we may apply Lemma \ref{le_vin_expsum} to conclude that
$$
\sum_{n \in I'} e\left( \frac{N}{n} + \frac{M}{n^j} \right) \ll X \log^{4A} X \exp( - c \log^{2\eps} X )$$
for some absolute constant $c>0$, and the claim follows.  If instead $F < X^4$, we can apply the Weyl inequality \cite[Theorem 8.4]{ik} with $k=5$ to conclude that
$$ \sum_{n \in I'} e\left( \frac{N}{n} + \frac{M}{n^j} \right) \ll \alpha^{O(1)} ( F/X^5 + 1/F )^c X \log X$$
for some absolute constant $c>0$; since $F \geq \log^{C A} X$, we obtain the claim by taking $C$ large enough.
\end{proof}

Now we prove Proposition~\ref{equid}(i).  We may assume without loss of generality that $j \geq 2$, since for $j=1$ we can absorb the $M$ terms into the $N$ term (and add a dummy term with $M=0$ and $j=2$, say).  By summation by parts (see e.g. \cite[Lemma 2.2]{mrt-div}), and adjusting $A$ as necessary, it suffices to show that
$$ \sum_{p \in I} e\left( \frac{N}{p} + \frac{M}{p^j} \right) \log p = \int_I e\left( \frac{N}{t} + \frac{M}{t^{j}}\right)\ dt + O_{\eps,A}( P\log^{-10A} P )$$
for all intervals $I \subset [P, P + P \log^{-100} P]$.  This is equivalent to
$$ \sum_{n \in I} e\left( \frac{N}{n} + \frac{M}{n^j} \right) \Lambda(n) = \int_I e\left( \frac{N}{t} + \frac{M}{t^{j}}\right)\ dt + O_{\eps,A}( P\log^{-10A} P ),$$
where $\Lambda$ is the von Mangoldt function,
since the contribution of the prime powers is negligible.  We introduce the quantity
$$ F \coloneqq \frac{|N|}{P} + \frac{|M|}{P^j}.$$
If $F \leq \log^{CA} P$
for some large absolute constant $C>0$, then the total variation of the phase $t \mapsto \frac{N}{t} + \frac{M}{t^{j}}$ is $O( \log^{CA} P)$, and the claim readily follows from a further summation by parts (see e.g. \cite[Lemma 2.2]{mrt-div}) and the prime number theorem (with classical error term).  Thus we may assume that
\begin{equation}\label{logcp}
 F > \log^{CA} P.
\end{equation}
In this case, a change of variables $t = P/s$ gives 
$$ \int_I e\left( \frac{N}{t} + \frac{M}{t^{j}}\right)\ dt = -P \int_{P/I} e\left( \frac{N}{P} s + \frac{M}{P^{j}} s^j \right)\ \frac{ds}{s^2}.$$
The derivative of the phase here is $N/P + j s^{j-1} M/P^j $ which, once $C$ is large enough, is $\geq \log^{10 A} P$ for all $s \in P/I$ apart from an interval of length at most $O(\log^{-10 A} P)$. Hence by partial integration we get that 
\[
\int_I e\left( \frac{N}{t} + \frac{M}{t^{j}}\right)\ dt \ll P \log^{-10A} P
\]
if $C$ is large enough, so it remains to establish the bound
$$ \sum_{n \in I} e\left( \frac{N}{n} + \frac{M}{n^j} \right) \Lambda(n) \ll P \log^{-10 A} P$$
under the hypothesis \eqref{logcp}.

By Vaughan's identity in the form of \cite[Proposition 13.4]{ik} (with $y=z=P^{1/3}$), followed by a shorter-than-dyadic decomposition, we can write 
\begin{align*}
\Lambda(n)=\sum_{r\leq R}(\alpha_r*1(n)+\alpha_r'*\log(n)+\beta_r*\gamma_r(n))
\end{align*}
for $n \in [P,2P]$, where $*$ denotes Dirichlet convolution, and
\begin{align*}
R&\ll \log^{O(1)} P,\\
|\alpha_r(n)|, |\alpha'_r(n)|, |\beta_r(n)|, |\gamma_r(n)|&\ll \log P,\\
\supp(\alpha_r), \supp(\alpha'_r) &\subset [M_r,(1 + \log^{-100} P)M_r], \\
\supp(\beta_r) &\subset [K_r,(1 + \log^{-100} P)K_r], \\
\supp(\gamma_r) &\subset [N_r,(1 + \log^{-100} P)N_r],\\
1 \leq M_r &\ll P^{2/3};\\
P^{1/3} \ll K_r, N_r &\ll P^{2/3}
\end{align*}
(the bound for the coefficients arising from Vaughan's identiy is $\ll \log P$ since $1 \ast \Lambda = \log$).  By the triangle inequality, it thus suffices to establish the Type I estimates
\begin{equation}\label{type-i}
\sum_{n \in I} e\left( \frac{N}{n} + \frac{M}{n^j} \right) (\alpha_r * 1)(n) \ll_{\eps,A} P \log^{-11 A} P
\end{equation}
and
\begin{equation}\label{type-i-alt}
\sum_{n \in I} e\left( \frac{N}{n} + \frac{M}{n^j} \right) (\alpha'_r * \log)(n) \ll_{\eps,A} P \log^{-11 A + 1} P
\end{equation}
as well as the Type II estimates
\begin{equation}\label{type-ii}
\sum_{n \in I} e\left( \frac{N}{n} + \frac{M}{n^j} \right) (\beta_r * \gamma_r)(n) \ll_{\eps,A} P \log^{-11 A} P
\end{equation}
for all $1 \leq r \leq R$ and $I \subset [P, P + P \log^{-100} P]$.  The second Type I estimate \eqref{type-i-alt} follows from the first Type I estimate \eqref{type-i} (replacing $\alpha_r$ with $\alpha'_r$) and a summation by parts (see e.g. \cite[Lemma 2.2]{mrt-div}), so it suffices to establish \eqref{type-i} and \eqref{type-ii}.

We begin with \eqref{type-i}.  By the triangle inequality, the left-hand side is bounded by
$$ \ll \log P \sum_{m \in [M_r,(1 + \log^{-100} P) M_r]} \left|\sum_{n \in \frac{1}{m} \cdot I} e\left( \frac{N}{mn} + \frac{M}{m^jn^j} \right)\right|.$$
Applying Proposition \ref{expi} with $X=P/m$ and $N/m$ and $M/m^j$ in place of $N$ and $M$, we can bound this by
$$ \ll_{\eps,A} P \log P \left( (1 + F)^c \log^{O(A)} P + \log^{-20 A} P \right)$$
for some constant $c>0$, and the claim now follows from \eqref{logcp}.

Now we establish \eqref{type-ii}.  We can assume that $K_r N_r \asymp P$, as the sum vanishes otherwise.  By the triangle inequality, the left-hand side is bounded by
$$ \ll \log P \sum_{m \in [K_r, (1 + \log^{-100} P) K_r]} \left|\sum_{n \in \frac{1}{m} \cdot I} \gamma_r(n) e\left( \frac{N}{mn} + \frac{M}{m^jn^j} \right)\right|.$$
By Cauchy--Schwarz it suffices to show that
$$ \sum_{m \in [K_r, (1 + \log^{-100} P) K_r]} \left|\sum_{n \in \frac{1}{m} \cdot I} \gamma_r(n) e\left( \frac{N}{mn} + \frac{M}{m^jn^j} \right)\right|^2 \ll K_r N^2_r \log^{-30 A} P$$
(say).  Rearranging, it suffices to show that
\begin{equation}
\label{eq:equidfinclaim} \sum_{n,n' \in [N_r, (1 + \log^{-100} P) N_r]} \gamma_r(n) \overline{\gamma_r(n')} X_{n,n'} \ll K_r N^2_r \log^{-30 A} P
\end{equation}
where
$$ X_{n,n'} \coloneqq \sum_{m \in [K_r,(1 + \log^{-100} P) K_r] \cap \frac{1}{n} \cdot I \cap \frac{1}{n'} \cdot I} e\left( \frac{N(n'-n)}{nn'm} + \frac{M((n')^j-n^j)}{n^j (n')^j m^j} \right).$$
By Proposition \ref{expi}, we have
$$ X_{n,n'} \ll_{\eps,A} K_r \left( \left(1 + \frac{|n'-n|}{N_r} F\right)^{-c} \log^{O(A)} P + \log^{-40A} P \right)$$
for some absolute constant $0 < c < 1$.  Bounding $\gamma_r(n) \overline{\gamma_r(n')} \ll \log^2 P$ and noting that
$$ \sum_{n \in [N_r, (1 + \log^{-100} P) N_r]} \left(1 + \frac{|n'-n|}{N_r} F\right)^{-c}
\ll N_r F^{-c}
$$
for all $n' \in [N_r, (1 + \log^{-100} P) N_r]$, we obtain the claim~\eqref{eq:equidfinclaim} from
\eqref{logcp}.  This completes the proof of Proposition \ref{equid}.
\section{Multiplicity of the falling factorial}\label{falling-sec}

In this section we establish Theorem \ref{main-falling}.  We first observe that if $1 \leq m \leq n$ solves \eqref{falling} for some sufficiently large $t$, then
$$ t = (n)_m \geq (m)_m = m! \gg (m/e)^m $$
by Stirling's formula. Hence we have an analogue of \eqref{m-bound}:
\begin{equation}\label{m-bound-falling}
 m \ll \frac{\log t}{\log_2 t}
\end{equation}
Next, since
$$ (n-m)^m < (n)_m \leq n^m$$
we have
\begin{equation}\label{tam}
 t^{1/m} \leq n < t^{1/m} + m
\end{equation}
and  we obtain an analogue 
\begin{equation}\label{n-form-falling}
 n \asymp t^{1/m} = \exp\left( \frac{\log t}{m} \right)
\end{equation}
of \eqref{n-form}, \eqref{n-form-alt}.

Next, we obtain the following analogue of Proposition \ref{distance}.

\begin{proposition}[Distance estimate]\label{distance-falling}  Suppose we have two solutions $(n,m), (n',m')$ to \eqref{falling} in region $\{ (m,n) \in \N^2: 1 \leq m \leq n \}$. Then one has
\begin{equation}\label{malt}
 m' - m \ll \log(n+n').
\end{equation}
Furthermore, if
\begin{equation}\label{mam}
 \exp( \log^{2/3+\eps}(n+n') ) \leq m,m' \leq (n+n')^{2/3}
\end{equation}
for some $\eps>0$, then we additionally have
\begin{equation}\label{nan}
 n' - n \ll_{A,\eps} \frac{m+m'}{\log^A(m+m')}
\end{equation}
for any $A>0$.
\end{proposition}

\begin{proof}  We begin with \eqref{malt}.  We follow the arguments from \cite[Proof of Theorem 4]{aeh}.  Taking $2$-valuations $v_2$ of both sides of \eqref{falling} and using \eqref{legendre} we have
$$ \sum_{j=1}^\infty \left(\left\lfloor \frac{n}{2^j}\right\rfloor - \left\lfloor \frac{n-m}{2^j}\right\rfloor\right) 
= \sum_{j=1}^\infty \left(\left\lfloor \frac{n'}{2^j}\right\rfloor - \left\lfloor \frac{n'-m'}{2^j}\right\rfloor\right).$$
The summands here vanish unless $j \leq \log(n+n')$.  Writing $\lfloor x \rfloor = x + O(1)$, we conclude that
$$ \sum_{1 \leq j \leq \log(n+n')} \frac{m}{2^j} + O(\log (n+n')) = \sum_{1 \leq j \leq \log(n+n')} \frac{m'}{2^j} + O(\log (n+n')) $$
and \eqref{malt} follows.

Now we prove \eqref{nan}.  Fix $A,\eps>0$. We may assume without loss of generality that $m' < m$, so that $n' > n$ by \eqref{falling}.  We may also assume $t$ is sufficiently large depending on $A, \eps$, as the claim is trivial otherwise; from \eqref{mam} this also implies that $m,m',n,n'$ are sufficiently large depending on $A,\eps$.  Henceforth all implied constants are permitted to depend on $A,\eps$.  By \eqref{mam} we have
$$ \log^{2/3+\eps} n \leq \log m$$
while from \eqref{n-form-falling} we have $\log n \asymp \frac{\log t}{m}$.  From this and \eqref{m-bound-falling} we have
\begin{equation}
\label{eq:msize}
m \asymp \frac{\log t}{\log_2 t}
\end{equation}
and then 
$$ \log n \ll \log^{\frac{1}{2/3 + \eps}}_2 t.$$
Similarly for $m',n'$.  From \eqref{malt} we conclude that
\begin{equation}\label{mmp}
 m-m' \ll \log^{\frac{1}{2/3 + \eps}}_2 t \ll \log^{\frac{1}{2/3+\eps}} m.
\end{equation}
In particular $m \asymp m'$ and, combining~\eqref{n-form-falling} with~\eqref{mmp} and~\eqref{eq:msize}, also $n \asymp n' t^{1/m - 1/m'} \asymp n'$. Hence from \eqref{mam} we see that
\begin{equation}\label{nbig}
 n, n' \gg m^{3/2}.
\end{equation}
Also we have
$$\frac{\log t}{m'} = \frac{\log t}{m} + O\left( \frac{\log t \log^{\frac{1}{2/3+\eps}} m}{m^2} \right) = \frac{\log t}{m} + O( m^{-1/2} )$$
(say), hence on exponentiating and using \eqref{tam}, \eqref{nbig}
\begin{equation}\label{stam}
 n' = \exp\left(\frac{\log t}{m'}\right) + O(m) = n + O( m^{-1/2} n ).
\end{equation}

Suppose that we could find a prime $p > m$ obeying the inequalities
\begin{equation}\label{obey}
\max\left(1 - \left\{ \frac{n'-n}{p}\right\}, 1 - \frac{m}{p}\right) < \left\{ \frac{n-m}{p} \right\} < 1; \quad \left\{ \frac{n'-n}{p}\right\} < 1 - \frac{m}{p}.
\end{equation}
These inequalities imply in particular that
\[
\left\{\frac{n-m}{p}\right\} - 1 + \frac{m}{p} \in [0, 1) \text{ and } \left\{\frac{n-m}{p}\right\} + \left\{ \frac{n'-n}{p} \right\} - 1 + \frac{m}{p} \in [0, 1),
\]
so that these quantities respectively equal $\{\frac{n}{p}\}$ and $\{\frac{n'}{p}\}$. Consequently, if~\eqref{obey} hold, then we would have 
\begin{equation}
\label{eq:obeycon1}
\left \{ \frac{n}{p}\right\} = \left\{\frac{n-m}{p}\right\} - 1 + \frac{m}{p} < \frac{m}{p}
\end{equation}
and (since $m' < m$)
\begin{equation}
\label{eq:obeycon2}
\left\{ \frac{n'}{p}\right\} = \left\{\frac{n-m}{p}\right\} + \left\{ \frac{n'-n}{p} \right\} - 1 + \frac{m}{p} \geq \frac{m}{p} \geq \frac{m'}{p}.
\end{equation}
Now~\eqref{eq:obeycon1} implies that $p$ divides $(n)_m$, while~\eqref{eq:obeycon2} implies that $p$ does not divide $(n')_{m'}$.  This contradicts the assumption $(n)_m = t = (n')_{m'}$.  Thus there cannot be any prime $p \geq 2m$ obeying \eqref{obey}.

Let $w_1 \colon \R \to [0,1]$ be a suitable smooth $\Z$-periodic function supported on the region
$$ \{ x \in \R: \{ x\} \in (1 - \log^{-2A} m,1)\}$$
chosen so that $\int_0^1 w_1 \gg \log^{-2A}m$ and $\|w_1\|_{C^3} \ll \log^{6A} m$,
and let $w_2 \colon \R \to [0,1]$ similarly be a smooth $\Z$-periodic function supported on the region
$$ \{ y \in \R: \{y\} \in (\log^{-2A} m, 1/2) \}$$
chosen so that $w_2(y) = 1$ when $\{y\} \in [2\log^{-2A} m, 1/4]$ and $\|w_2\|_{C^3} \ll \log^{6A}m$.  Let $\p$ be a prime drawn uniformly from all the primes in $[2m, 100m]$.  As $\p$ does not obey \eqref{obey}, we have
$$ \E w_1\left(\frac{n-m}{\p}\right) w_2\left(\frac{n'-n}{\p}\right) = 0$$
and hence by Proposition \ref{equid} (and dyadic decomposition)
$$ \int_2^{100} w_1\left( \frac{n-m}{tm}\right) w_2\left(\frac{n'-n}{tm} \right)\ dt \ll \log^{-100 A} m,$$
or on changing variables $t=1/s$
\begin{equation}\label{w1w2} 
\int_{1/100}^{1/2} w_1\left( \frac{n-m}{m}s\right) w_2\left(\frac{n'-n}{m}s \right)\ ds \ll \log^{-100 A} m.
\end{equation}
On the other hand, by \eqref{stam}, \eqref{nbig} we have
\begin{equation} \label{w1large}
\frac{n-m}{m} \gg \frac{n}{m} \gg m^{1/2} \frac{n'-n}{m} + m^{1/2}.
\end{equation}
We perform a Fourier expansion 
$$ w_1(x) = \sum_{\ell \in \Z} c_{\ell} e(\ell x), $$
where by integration by parts the Fourier coefficients obey the bounds
$$ |c_{\ell}| \ll (1 + |\ell|)^{-3} \log^{6A}m. $$
Thus~\eqref{w1w2} can then be rewritten as
\begin{equation}\label{w1w2fourier} 
\sum_{\ell \in \Z} c_{\ell} \int_{1/100}^{1/2} w_2\left(\frac{n'-n}{m}s \right) e\left(\frac{n-m}{m} \ell s\right)\ ds \ll \log^{-100A}m.
\end{equation}
By~\eqref{w1large} and integration by parts, one readily establishes the bound
$$ \int_{1/100}^{1/2} w_2\left(\frac{n'-n}{m}s \right) e\left(\frac{n-m}{m} \ell s\right)\ ds \ll \frac{\log^{6A} m}{|\ell| m^{1/2}} $$
for $\ell \neq 0$. Thus the total contribution to the left-hand side of~\eqref{w1w2fourier} from the terms with $\ell \neq 0$ is negligible, and hence
$$ c_0 \int_{1/100}^{1/2} w_2\left(\frac{n'-n}{m}s \right)\ ds \ll \log^{-100A}m.  $$
Since $c_0 = \int_0^1 w_1 \gg \log^{-2A}m$ and $w_2$ equals $1$ on $[2\log^{-2A}m, 1/4]$, we have
\begin{equation}
\label{eq:fupper} 
f\left( \frac{n'-n}{m} \right) \ll \log^{-98 A} m
\end{equation}
where
$$ f(\theta) \coloneqq \int_{1/100}^{1/2} 1_{2 \log^{-2A} m \leq \{ \theta s\} \leq 1/4} \ ds.$$
However, direct calculation shows that when $\theta \geq 3$, we have
$$ f(\theta) \geq \sum_{\frac{\theta}{16} \leq n \leq \frac{\theta}{2}-\frac{1}{4}} \int_\R 1_{n + 1/100 \leq \theta s \leq n+1/4}\ ds \gg \theta \cdot \theta^{-1} = 1,$$
when $1/2 < \theta < 3$, we have
$$ f(\theta) \geq \int_{\frac{1}{30\theta}}^{\frac{1}{20\theta}}\ ds \asymp 1,$$
and, when $8\log^{-A} m \leq \theta \leq 1/2$, we have
$$ f(\theta) \geq \int_{1/4}^{1/2}\ ds \asymp 1.$$
Hence~\eqref{eq:fupper} can only hold if
$$ \frac{n'-n}{m} \ll \log^{-A} m,$$
giving the claim \eqref{nan}.
\end{proof}

Now we adapt the analysis from Section \ref{analytic-sec}.  We extend the falling factorial $(n)_m$ to real $n \geq m \geq 0$ by the formula
$$ (n)_m \coloneqq \frac{\Gamma(n+1)}{\Gamma(n-m+1)}.$$
From the increasing nature of the digamma function $\psi$ we see that for fixed $m$, $(n)_m$ increases from $\Gamma(m+2)$ when $n$ goes from $m+1$ to infinity. Applying the inverse function theorem, we conclude that for any sufficiently large $t$ there is a unique smooth function $g_t \colon \{ m > 0: \Gamma(m+2) \leq t \} \to \R$ such that for any $m>0$ with $\Gamma(m+2) \leq t$, one has $g_t(m) \geq m$ and
\begin{equation}\label{gtm}
(g_t(m))_m = t.
\end{equation}
Indeed, one could simply set $g_t(m) \coloneqq f_{t/\Gamma(m+1)}(m)$, where $f_t$ is the function studied in Section \ref{analytic-sec}.

We have an analogue of Proposition \ref{derivi}:

\begin{proposition}[Estimates on the first few derivatives]\label{derivi-falling} Let $C > 1$, and let $t,m$ be sufficiently large depending on $C$ with $\Gamma(m+2) \leq t$.  
Then
\begin{equation}\label{fatm-falling}
 g_t(m) \asymp t^{1/m}.
\end{equation}
In the range $m \leq g_t(m)/2$, we have
\begin{equation}\label{one-falling}
-g'_t(m) \asymp g_t(m)  \frac{\log t}{m^2}
\end{equation}
and in the range $m \leq g_t(m) - C \log^2 g_t(m)$, one has
\begin{equation}\label{ten-falling}
0 < g''_t(m) \ll g_t(m) \left(\frac{\log t}{m^2}\right)^2 + C^{-1} \log^{-3} m.
\end{equation}
\end{proposition}

\begin{proof}  Write $n=g_t(m) \geq m$. First note that~\eqref{fatm-falling} is simply~\eqref{n-form-falling}.
Taking logarithms in \eqref{gtm} we have
\begin{equation}\label{logf-falling}
 \log \Gamma(g_t(m)+1) - \log \Gamma(g_t(m)-m+1) = \log t.
\end{equation}

If we differentiate \eqref{logf-falling} we obtain
\begin{equation}\label{deriv-falling}
 g'_t(m) \psi( g_t(m)+1) - (g'_t(m)-1) \psi(g_t(m)-m+1) = 0.
\end{equation}
In particular we obtain the first derivative formula
\begin{equation}\label{form-falling}
g'_t(m) = \frac{- \psi(n-m+1)}{\psi(n+1) - \psi(n - m + 1)}.
\end{equation}
In the regime $m \leq n/2$ we can then obtain \eqref{one-falling} from \eqref{psini}, \eqref{psi-1}, \eqref{fatm-falling}. 

Differentiating \eqref{deriv-falling} again, we conclude
$$
 g''_t(m) \psi( n+1) + (g'_t(m))^2 \psi'( n+1) - g''_t(m) \psi(n-m+1) 
-(g'_t(m)-1)^2 \psi'(n-m+1) = 0
$$
which we can rearrange using \eqref{form-falling} as
\begin{equation}
\label{eq:g''tform}
\begin{split}
g''_t(m) (\psi(n+1) - \psi(n-m+1))^3 &=  \psi(n+1)^2 \psi'(n-m+1)\\
&\quad - \psi(n-m+1)^2 \psi'(n+1).
\end{split}
\end{equation}
Suppose first that $m \leq n/2$.  Then \eqref{psini} applies, and it suffices to show that
$$ \psi(n+1)^2 \psi'(n-m+1) - \psi(n-m+1)^2 \psi'(n+1) \asymp \left( \frac{m}{n} \right)^3 n \left(\frac{\log t}{m^2}\right)^2.$$
By \eqref{fatm-falling} the right-hand side is $\asymp \frac{m \log^2 n}{n^2}$.  On the other hand, from the mean value theorem and \eqref{psi-1}, \eqref{psi-2}, \eqref{psi-3} we have
$$ 0 < \psi(n+1)^2 (\psi'(n-m+1) - \psi'(n+1)) \asymp \frac{m \log^2 n}{n^2}$$
and
$$ 0 < (\psi(n+1)^2 - \psi(n-m+1)^2) \psi'(n+1) \ll \frac{m \log n}{n^2}$$
giving the claim.

Now suppose that $n/2 \leq m \leq n - C \log^2 n$.  From \eqref{psi-1}, \eqref{psi-2} we have
\begin{align*}
\psi(n+1) - \psi(n-m+1) &\asymp \log \frac{n}{n-m} \\
\psi(n+1)^2 (\psi'(n-m+1) - \psi'(n+1)) 
&\asymp \frac{\log^2 n}{n-m} \\
0 < (\psi(n+1)^2 - \psi(n-m+1)^2) \psi'(n+1) &\ll \frac{\log^2 n}{n} 
\end{align*}
and hence by~\eqref{eq:g''tform}
$$ g_t''(m) \asymp \frac{\log^2 n}{(n-m) \log^3 \frac{n}{n-m}}.$$
Since $\frac{n}{2} \geq n-m \geq C \log^2 n$, we have
$$ (n-m) \log^3 \frac{n}{n-m} \gg C \log^5 n$$
(as can be seen by checking the cases $n-m \leq \sqrt{n}$ and $n-m > \sqrt{n}$ separately), and the claim follows.
\end{proof}

Now we can establish Theorem \ref{main-falling}.  Let $C>0$ be a large absolute constant, let $\eps>0$, and suppose that $t$ is sufficiently large depending on $\eps,C$.  Let $(n,m)$ be the integer solution to \eqref{falling} in the region $\exp( \log^{2/3 + \eps} n ) \leq m \leq n- 1$ with a maximal value of $m$; we may assume that such a solution exists, since we are done otherwise.  If $(n',m')$ is any other solution in this region, then $m'<m$ and $n < n'$.  Note that $n,n',m,m'$ are sufficiently large depending on $\eps,C$.  From Proposition \ref{distance-falling} and \eqref{n-form-falling} we have
$$ m - m' \ll \log n' \ll \frac{\log t}{m'}$$
and from~\eqref{n-form-falling} and~\eqref{m-bound-falling}
$$m' \asymp \frac{\log t}{\log n'} \geq \frac{\log t}{\log^{\frac{1}{2/3+\eps}} m'} \gg \frac{\log t}{\log^{\frac{1}{2/3+\eps}}_2 t}$$
and thus $m' \asymp m$ and $\frac{\log t}{m} = \frac{\log t}{m'} + O(1)$. Hence $n \asymp n'$  and
\begin{equation}\label{mamn}
 m - m' \ll \log n.
\end{equation}

First suppose that $m \leq n^{1/2} \log^{10} n$.  Here we will exploit the fact that $n$ grows rapidly as $m$ decreases.  From Proposition \ref{distance-falling} we have
$$ n' - n \ll_{\eps} \frac{m}{\log^{200} m} \ll \frac{m}{\log^{100} n}.$$
On the other hand, from \eqref{one-falling} and the mean value theorem we have
$$ n' - n = g_t(m') - g_t(m) \gg \frac{n \log t}{m^2} (m-m') \geq \frac{n}{m}$$
thanks to \eqref{m-bound-falling} and the trivial bound $m-m' \geq 1$.  Thus we have
$$ \frac{n}{m} \ll \frac{m}{\log^{100} n}$$
but this contradicts the hypothesis $m \leq n^{1/2} \log^{10} n$.  

Now suppose we are in the regime
$$ n^{1/2} \log^{10} n < m \leq n - C \log^2 n.$$
Here we will take advantage of the convexity properties of $g_t$.
From \eqref{mamn}, $m'$ lies in the interval $[m - O(\log n), m]$.  By \eqref{fatm-falling}, for all $x$ in this interval, we have
$$ g_t(x) \asymp t^{1/x} \asymp t^{1/m} \asymp n$$
and by \eqref{ten-falling}, we have
\begin{align*}
 0 < g''_t(x) &\ll g_t(x) \left(\frac{\log t}{x^2}\right)^2 + C^{-1} \log^{-3} x \\
&\ll n \left(\frac{\log t}{m^2}\right)^2 + C^{-1} \log^{-3} m \\
&\ll n \left(\frac{\log n}{m}\right)^2 + C^{-1} \log^{-3} n \\
&\ll C^{-1} \log^{-3} n
\end{align*}
since $m > n^{1/2} \log^{10} n$.  Applying Lemma \ref{integ} with $k = 2$, we see (for $C$ large enough) that there are at most two integers $m'$ in this interval with $g_t(m')$ an integer, giving Theorem \ref{main-falling} follows in this case.

It remains to handle the case
\begin{equation}\label{nmc}
 n - C \log^2 n < m \leq n-1.
\end{equation}
Recall from~\eqref{mamn} that $m'$ lies in the interval $[m - O(\log n), m]$.  From \eqref{n-form-falling}, \eqref{nmc} we have
$$ m \asymp n \asymp \frac{\log t}{\log_2 t}$$
so $m' = m - O(\log_2 t)$.  From \eqref{n-form-falling} again we thus also have
$$ m' \asymp n' \asymp \frac{\log t}{\log_2 t}.$$
From \eqref{falling} we have
$$ \frac{n'}{n'-m'} \frac{n'-1}{n'-1-m'} \dots \frac{n+1}{n+1-m'} = (n-m') \dots (n-m+1).$$
The right-hand side is at most $\exp(O( \log_2 t \log_3 t) )$. This implies that $n' - n \ll \log_3 t$, since otherwise the left hand side would be, for any $C \geq 1$,
\[
\gg \left(\frac{n}{n-m'+1+ C \log_3 t}\right)^{C \log_3 t} \gg \exp\left(\frac{C}{2} \log_3 t \log_2 t\right)
\]
which contradicts the bound for the right hand side when $C$ is sufficiently large.

In particular we have from the triangle inequality that
$$ n-m, n' - m' \ll C \log^2_2 t.$$
Making the change of variables $\ell := n-m$, it now suffices to show that there are at most two integer solutions to the equation
\begin{equation}\label{l-eq}
(n)_{n-\ell} = t
\end{equation}
in the regime $1 \leq \ell \ll C \log^2_2 t$.  We write this equation \eqref{l-eq} as
$$ n! = t \ell!$$
or equivalently
$$ n = h_t(\ell)$$
where $h_t(x) \coloneqq \Gamma^{-1}( t \Gamma(x+1 ) ) - 1 $, and $\Gamma^{-1} \colon [1,+\infty) \to [2,+\infty)$ is the inverse of the gamma function.  Here we will exploit the very slowly varying nature of $h_t$.  From Stirling's formula we have
$$ h_t(x) \asymp \frac{\log t}{\log_2 t}$$
whenever $1 \leq x \ll C \log^2_2 t$.
Taking the logarithmic derivative of the equation
$$ \Gamma(h_t(x) + 1) = t \Gamma(x+1)$$
we have
$$ h'_t(x) \psi(h_t(x)+1) = \psi(x+1).$$
Hence by \eqref{psi-1}
$$ h'_t(x) \asymp \frac{\log x}{\log h_t(x)} \ll \frac{\log_3 t}{\log_2 t}$$
in the regime $1 \leq x \ll C \log^2_2 t$.  In particular, for two solutions $(n,\ell), (n',\ell')$ to \eqref{l-eq} in this regime we have
\begin{equation}\label{nnp}
 n-n' \ll \frac{\log_3 t}{\log_2 t} |\ell-\ell'|.
\end{equation}
For fixed $n$ there is at most one $\ell \geq 1$ solving \eqref{l-eq}.  We conclude that for two distinct solutions $(n,\ell), (n',\ell')$ to \eqref{l-eq} in this regime, we have $|n-n'| \geq 1$, and hence the separation
$$ |\ell-\ell'| \gg \frac{\log_2 t}{\log_3 t}.$$

Now suppose we have three solutions $(n_1,\ell_1), (n_2,\ell_2), (n_3,\ell_3)$ to \eqref{l-eq} in this regime.  We can order $\ell_1 < \ell_2 < \ell_3$, so that $n_1 < n_2 < n_3$.  From the preceding discussion we have
$$ \frac{\log_2 t}{\log_3 t} \ll \ell_2 - \ell_1, \ell_3 - \ell_2 \ll C \log^2_2 t$$
and
$$ 1 \leq n_2-n_1, n_3 - n_2 \ll C \log_2 t \log_3 t.$$
If $2^j$ is a power of $2$ that divides an integer in $(n_1,n_2]$ as well as an integer in $(n_2,n_3]$, then we must therefore have $2^j \ll C \log_2 t \log_3 t$, so that $j \ll \log_3 t$.  Thus, there must exist $i=1,2$ such that the interval $(n_i,n_{i+1}]$ only contains multiples of $2^j$ when $j \ll \log_3 t$.  Fix this $i$.  Taking $2$-adic valuations of \eqref{l-eq} using \eqref{legendre} we have
$$ \sum_{j=1}^{\infty} \left\lfloor \frac{n_i}{2^j} \right \rfloor = v_2(t) + \sum_{j=1}^\infty \left\lfloor \frac{\ell_i}{2^j} \right \rfloor $$
and
$$ \sum_{j=1}^{\infty} \left\lfloor \frac{n_{i+1}}{2^j} \right \rfloor = v_2(t) + \sum_{j=1}^\infty \left\lfloor \frac{\ell_{i+1}}{2^j} \right \rfloor $$
and thus
\begin{equation}
\label{eq:2adiccomp}
\sum_{j=1}^{\infty} \left(\left\lfloor \frac{n_{i+1}}{2^j} \right \rfloor - \left\lfloor \frac{n_{i}}{2^j} \right \rfloor\right)  = \sum_{j=1}^\infty \left(\left\lfloor \frac{\ell_{i+1}}{2^j} \right \rfloor - \left\lfloor \frac{\ell_{i}}{2^j} \right \rfloor\right).
\end{equation}
Since 
\begin{equation}\label{lii}
\ell_{i+1} - \ell_i \gg \frac{\log_2 t}{\log_3 t},
\end{equation}
we certainly have $\ell_{i+1} - \ell_i \geq 2$, and the right-hand side of~\eqref{eq:2adiccomp} is at least
$$ \left\lfloor \frac{\ell_{i+1}}{2} \right \rfloor - \left\lfloor \frac{\ell_{i}}{2} \right \rfloor \gg \ell_{i+1} - \ell_i.$$
By construction, the terms on the left-hand side of~\eqref{eq:2adiccomp} vanish unless $j \ll \log_3 t$, in which case they are equal to $\frac{n_{i+1}-n_i}{2^j} + O( 1)$.  Thus the left-hand side of~\eqref{eq:2adiccomp} is at most $O( n_{i+1} - n_i + \log_3 t )$.  Thus
$$
\ell_{i+1} - \ell_i \ll n_{i+1} - n_i + \log_3 t.$$
But from \eqref{nnp} one has $n_{i+1} - n_i  \ll \frac{\log_3 t}{\log_2 t} (\ell_{i+1}-\ell_i)$. Hence $\ell_{i+1} - \ell_i \ll \log_3 t$.  But this contradicts \eqref{lii}.  This concludes the proof of Theorem \ref{main-falling}. 
\bibliography{refs}
\bibliographystyle{plain}

\end{document}